\documentclass{msc}

\usepackage{enumitem} 
\setenumerate{topsep=0pt} 
\usepackage{manfnt} 
\usepackage{marginnote} 

\usepackage{mathtools}
\usepackage{stmaryrd} 

\usepackage{savesym}
\savesymbol{gtrsim}
\savesymbol{lesssim}
\usepackage{amssymb}

\usepackage{epsfig}

\theoremstyle{definition}
\newtheorem{example}[therm]{Example}

\newtheorem{warning}[therm]{Warning}

\DeclarePairedDelimiter{\pa}{(}{)}
\DeclarePairedDelimiter{\set}{\{}{\}}

\newcommand{\Nat}{\mathbb N}
\newcommand{\Rat}{\mathbb Q}
\newcommand{\Rea}{\mathbb R}
\newcommand{\Two}{\mathbf{2}}

\newcommand{\Sierp}{\mathbb S}

\newcommand{\Seqspace}{\mathbf A}
\newcommand{\Seqdom}{\mathcal A}
\newcommand{\Finseq}{A^\ast}

\newcommand{\Realdom}{\mathcal R}
\newcommand{\LowerReal}{\mathcal L}

\DeclareMathOperator{\nbhds}{\mathcal N}

\newcommand{\below}{\mathrel{\sqsubseteq}}
\newcommand{\sbelow}{\mathrel{\sqsubset}}
\newcommand{\aboveorder}{\mathrel{\sqsupseteq}}
\newcommand{\dirsup}{\bigsqcup}
\newcommand{\glb}{\mathrel{\sqcap}}

\DeclareMathOperator{\dset}{\downarrow}
\DeclareMathOperator{\upset}{\uparrow}
\newcommand\twoheaddownarrow{\rotatebox[origin=c]{270}{\(\twoheadrightarrow\)}}
\DeclareMathOperator{\ddset}{\twoheaddownarrow}
\newcommand\twoheaduparrow{\rotatebox[origin=c]{90}{\(\twoheadrightarrow\)}}
\DeclareMathOperator{\upupset}{\twoheaduparrow}

\DeclarePairedDelimiter{\steppa}{\llparenthesis}{\rrparenthesis}

\newcommand{\apart}{\mathrel{\#}}
\newcommand{\posnotbelow}{\mathrel{\;\not\!\not{\!\leq}}}
\newcommand{\apartvar}{\mathrel{\sharp}}

\newcommand{\lcompl}{{\lnot}\hspace{0.12em}}
\newcommand{\compl}{{\sim}}
\newcommand{\acompl}{-}
\newcommand{\setapart}{\mathrel{\bowtie}}

\DeclareMathOperator{\length}{length}
\newcommand{\initseg}[2]{\bar{#1}_{#2}}
\DeclarePairedDelimiter{\seq}{\langle}{\rangle}

\DeclareMathOperator{\Idl}{Idl}
\newcommand{\powerset}[1]{\mathcal P\pa*{#1}}
\newcommand{\To}{\Rightarrow}
\newcommand{\interior}[1]{\pa*{#1}^\circ}
\newcommand{\refined}{\mathrel{\upupset}}
\DeclareMathOperator{\SMax}{StrongMax}

\DeclareMathOperator{\lifting}{\mathcal L}
\DeclareMathOperator{\image}{image}

\newcommand{\arxivlink}[1]{\href{https://arxiv.org/pdf/#1}{\texttt{arXiv:#1}}}
\renewcommand{\doilink}[1]{\href{https://doi.org/#1}{\texttt{doi:#1}}}

\begin{document}

\lefttitle{T. de Jong} \righttitle{Apartness, Sharp Elements and the Scott
  Topology of Domains} \title{Apartness, Sharp Elements and the Scott Topology
  of Domains\footnote{A shorter version of this paper appeared as \emph{Sharp
      Elements and Apartness in Domains}. In A. Sokolova (ed.)  \emph{37th
      Conference on Mathematical Foundations of Programming Semantics (MFPS
      2021)}, {Electronic Proceedings in Theoretical Computer Science (EPTCS)},
    vol.~351, Open Publishing Association, 134--151.
    \doilink{10.4204/EPTCS.351.9}.}}

\papertitle{Paper}

\jnlPage{1}{23}
\jnlDoiYr{2023}
\doival{10.1017/S0960129523000282}

\begin{authgrp}
\author{Tom de Jong}
\affiliation{School of Computer Science, University of Birmingham,
             Birmingham, UK \\
             Email: \email{tom.dejong@nottingham.ac.uk}}
\end{authgrp}


\begin{abstract}
  Working constructively, we study continuous directed complete posets (dcpos)
  and the Scott topology.
  Our two primary novelties are a notion of intrinsic apartness and a notion of
  sharp elements.
  Being apart is a positive formulation of being unequal, similar to how
  inhabitedness is a positive formulation of nonemptiness.
  To exemplify sharpness, we note that a lower real is sharp if and only if it
  is located.
  Our first main result is that for a large class of continuous dcpos, the
  Bridges--{V\^i\c{t}\v{a}} apartness topology and the Scott topology coincide.
  Although we cannot expect a tight or cotransitive apartness on nontrivial
  dcpos, we prove that the intrinsic apartness is both tight and cotransitive
  when restricted to the sharp elements of a continuous dcpo.
  These include the strongly maximal elements, as studied by Smyth and
  Heckmann. We develop the theory of strongly maximal elements highlighting its
  connection to sharpness and the Lawson topology.
  Finally, we illustrate the intrinsic apartness, sharpness and strong
  maximality by considering several natural examples of continuous dcpos:
  the Cantor and Baire domains, the partial Dedekind reals, the lower reals and
  finally, an embedding of Cantor space into an exponential of lifted sets.
\end{abstract}


\maketitle

\section{Introduction}\label{sec:introduction}
Domain theory~\citep{AbramskyJung1995} is rich with applications in semantics of
programming languages~\citep{Scott1993,Scott1982,Plotkin1977}, topology and
algebra~\citep{GierzEtAl2003}, and higher-type
computability~\citep{LongleyNormann2015}. The basic objects of domain theory are
directed complete posets (dcpos), although we often restrict our attention to
algebraic or continuous dcpos which are generated by so-called compact elements
or, more generally, by the so-called way-below relation
(Section~\ref{sec:preliminaries}).
We examine the Scott topology on dcpos using an apartness relation and a notion
of sharp elements. Our work is constructive in the sense that we do not assume
the principle of excluded middle or choice axioms, but we do use impredicativity
and powersets in particular, see Section~\ref{sec:foundations} for more details
on our foundational setup.

Classically, i.e.\ when assuming excluded middle, a dcpo with the Scott topology
satisfies \(T_0\)-separation: if two points have the same Scott open
neighbourhoods, then they are equal.
This holds constructively if we restrict to continuous dcpos.
A~classically equivalent formulation of \(T_0\)-separation is: if \(x \neq y\),
then there is a Scott open separating \(x\) and \(y\), i.e.\ containing \(x\)
but not \(y\) or vice versa.
This second formulation is equivalent to excluded middle.
This brings us to the first main notion of this paper
(Section~\ref{sec:intrinsic-apartness}). We say that \(x\)~and~\(y\) are
intrinsically apart, written \(x \apart y\), if there is a Scott open containing
\(x\) but not \(y\) or vice versa. Then \(x \apart y\) is a positive formulation
of \(x \neq y\), similar to how inhabitedness (i.e.\ \(\exists\,{x \in X}\)) is
a positive formulation of nonemptiness (i.e.\ \(X \neq \emptyset\)).

This definition works for any dcpo, but the intrinsic apartness is mostly of
interest to us for continuous dcpos. In fact, the apartness really starts to
gain traction for continuous dcpos that have a basis satisfying certain
decidability conditions.
For example, we prove that for such continuous dcpos, the apartness
topology~\citep{BridgesVita2011} and the Scott topology coincide
(Section~\ref{sec:apartness-topology}).
Thus our work may be regarded as showing that the constructive framework by
Bridges and V\^i\c{t}\v{a} is applicable to domain theory.
It should be noted that these decidability conditions are satisfied by the major
examples in applications of domain theory to topology and computation. Moreover,
these conditions are stable under products of dcpos and, in the case of bounded
complete algebraic dcpos, under exponentials (Section~\ref{sec:preliminaries}).

In \cite[p.~7]{BridgesRichman1987}, \cite[p.~8]{MinesRichmanRuitenburg1988} and
\cite[p.~8]{BridgesVita2011}, an irreflexive and symmetric relation is called an
{inequality (relation)} and the symbol~\({\neq}\) is used to denote it. In
\cite[Definition~2.1]{BishopBridges1985}, an inequality is moreover required to
be cotransitive:
\begin{center}
  if \(x \neq y\), then \(x \neq z\) or \(y \neq z\) for any
    \(x\), \(y\) and \(z\).
\end{center}
The latter is called a {preapartness} in
\cite[Section~8.1.2]{TroelstraVanDalen1988} and the symbol~\({\apart}\) is used
to denote it, reserving \({\neq}\) for the logical negation of equality and the
word {apartness} for a relation that is also tight: if \(\lnot(x \apart y)\),
then \(x = y\).

\begin{warning}\label{warning}
  \marginnote{\dbend}
  We deviate from the above and use the word apartness and the symbol
  \({\apart}\) for an irreflexive and symmetric relation, so we do not require
  it to be cotransitive or tight.
\end{warning}
The reasons for our choice of terminology and notations are as follows: (i) we
wish to reserve \(\neq\) for the negation of equality as in
\cite[Section~8.1.2]{TroelstraVanDalen1988}; (ii) the word inequality is
confusingly also used in the context of posets to refer to the partial order;
and finally, (iii) the word inequality seems to suggest that the negation of the
inequality relation is an equivalence relation, but, in the absence of
cotransitivity, it need not be.

Actually, we prove that no apartness on a nontrivial dcpo can be
cotransitive or tight unless (weak) excluded middle holds.
However, there is a natural collection of elements for which the intrinsic
apartness is both tight and cotransitive: the sharp elements
(Section~\ref{sec:tightness-cotransitivity-sharpness}). Sharpness is slightly
involved in general, but it is easy to understand for algebraic dcpos: an
element \(x\) is sharp if and only if for every compact element it is decidable
whether it is below \(x\).
Moreover, the notion is quite natural in many examples. For instance, the sharp
elements of a powerset are exactly the decidable subsets and the sharp lower
reals are precisely the located ones.

An import class of sharp elements is given by the strongly maximal elements
(Section~\ref{sec:strongly-maximal-elements}). These were studied in a classical
context in~\citep{Smyth2006} and~\citep{Heckmann1998}, because of
their desirable properties. For instance, while the subspace of maximal elements
may fail to be Hausdorff, the subspace of strongly maximal elements is both
Hausdorff (two distinct points can be separated by disjoint Scott opens) and
regular (every neighbourhood contains a Scott closed neighbourhood).
As shown by \cite{Smyth2006}, strong maximality is closely related to the Lawson
topology. Specifically, Smyth proved that a point \(x\) is strongly maximal
if and only if every Lawson neighbourhood of \(x\) contains a Scott
neighbourhood of \(x\).
Using sharpness, we offer a constructive proof of~this.

Finally, in Section~\ref{sec:examples} we illustrate the above notions by
presenting natural examples of continuous dcpos, many of which embed well-known
spaces as the strongly maximal elements. Specifically, we consider
the Cantor and Baire domains, the partial Dedekind reals, the lower reals, and
an embedding of Cantor space into an exponential of lifted sets.

\subsection{Foundations}\label{sec:foundations}
We work informally in an impredicative set theory with intuitionistic logic. If
we wish to pin down a particular formal system, then we could choose
IZF~\citep{Friedman1973}, which has models in Grothendieck and realizability
toposes \citep{Fourman1980,Hayashi1981,JoyalMoerdijk1995}.
In particular we have powersets and allow for the usual impredicative
definitions of interior and closure in topology.
Functions, as is standard in set theory, are encoded as single-valued and total
relations.
We stress that our work is fully compatible with classical logic.

\subsection{Contributions}
Our primary objective is to explore the Scott topology of continuous dcpos in a
constructive setting where we do not rely on excluded middle or the axiom of
choice. This leads us to study an apartness relation, as well as the subsets of
sharp and strongly maximal elements, as explained in the introduction.

We list some main results: Theorem~\ref{apartness-Scott-topology} says the
framework of~\cite{BridgesVita2011} is applicable to constructive domain theory;
Theorem~\ref{sharp-tight-cotransitive} shows that the intrinsic apartness is
tight and cotransitive for sharp elements. We have a supply of sharp elements
thanks to Theorems~\ref{sharp-basis} and \ref{sharpness-product-exponential}, as
well as Proposition~\ref{sharp-if-strongly-maximal} which tells us that strongly
maximal elements are sharp.
Finally, the examples of Section~\ref{sec:examples} provide important
illustrations of the theory. For example, Theorem~\ref{Dedekind-iota-homeo}
shows that the real line is homeomorphic to the subspace of strongly maximal
elements in the continuous domain of partial Dedekind reals, and moreover, that
the homeomorphism reflects and preserves apartness.

Some notions, such as sharpness, trivialize in a classical setting. However, our
simplification of Smyth's definition of strong maximality
(Section~\ref{sec:strongly-maximal-elements}) and
Proposition~\ref{strong-maximality-product-exponential} are new contributions in
a classical setting as well.
Moreover, our constructive treatment could possibly inform a classical, but
effective treatment of domain theory~\citep{Smyth1977}.

\subsection{Related work}
There are numerous accounts of basic domain theory in several constructive
systems, such as
\citep{SambinValentiniVirgili1996,Negri1998,Negri2002,MaiettiValentini2004,Kawai2017,Kawai2021}
in the predicative setting of formal
topology~\citep{Sambin1987,CoquandEtAl2003}, as well as works in various type
theories: \citep{Hedberg1996} in (a version of) Martin-L\"of Type Theory,
\citep{Lidell2020} in Agda, \citep{BentonKennedyVarming2009,Dockins2014} in Coq
and our previous work~\citep{deJongEscardo2021a,deJong2022} in univalent
foundations. Besides that, the
papers~\citep{BauerKavkler2009,PattinsonMohammadian2021} are specifically aimed
at program extraction.

Our work is not situated in formal topology and we work informally in
impredicative set theory without using excluded middle or choice axioms. We
also consider completeness with respect to all directed subsets and not just
\(\omega\)-chains as is done
in~\citep{BauerKavkler2009,PattinsonMohammadian2021}. The principal
contributions of our work are the aforementioned notions of intrinsic apartness
and sharp elements, although the idea of sharpness also appears in formal
topology: an element of a continuous dcpo is sharp if and only if its filter of
Scott open neighbourhoods is located in the sense of \cite{Spitters2010} and
\cite{Kawai2017}.

If, as advocated in \citep{Abramsky1987,Vickers1989,Smyth1993}, we think of
(Scott) opens as observable properties, then this suggests that we label two
points as apart if we have made conflicting observations about them, i.e.\ if
there are disjoint opens separating the points. Indeed, (an equivalent
formulation of) this notion is used in \cite[p.~362]{Smyth2006}.
While these notions are certainly useful, both in the presence and absence of
excluded middle, our apartness serves a different purpose: It is a positive
formulation of the negation of equality used when reasoning about the Scott
topology on a dcpo, which (classically) is only a \(T_0\)-space that isn't
Hausdorff in general. By contrast, an apartness based on disjoint opens would
supposedly perform a similar job for a Hausdorff space, such as a dcpo with the
Lawson topology.

Finally, \cite{vonPlato2001} gives a constructive account of so-called positive
partial orders: sets with a binary relation \({\not\leq}\) that is irreflexive
and cotransitive (i.e.\ if \(x \not\leq y\), then \(x \not\leq z\) or
\(y \not\leq z\) for any elements \(x\), \(y\) and~\(z\)).
Our notion \({\posnotbelow}\) from Definition~\ref{def:specialization} bears
some similarity, but our work is fundamentally different for two
reasons. Firstly, \({\posnotbelow}\) is not cotransitive. Indeed, we cannot
expect such a cotransitive relation on nontrivial dcpos, cf.\
Theorem~\ref{tight-cotransitive-wem}. Secondly, in \citep{vonPlato2001} equality
is a derived notion from \({\not\leq}\), while equality is primitive for us.

\subsection{Acknowledgements}
I am very grateful to Mart\'in Escard\'o for many discussions (including one
that sparked this paper) and valuable suggestions for improving the
exposition. In particular, the terminology ``sharp'' is due to Mart\'in and
Theorem~\ref{sharp-iff-located} was conjectured by him.
I~should also like to thank Steve Vickers for his interest and remarks.
Furthermore, I thank the anonymous referees of the shorter, related MFPS paper
for their helpful comments and questions.
Finally, I am grateful to the anonymous referee whose comments and questions
have led to the addition of
Theorem~\ref{sharpness-product-exponential},
Proposition~\ref{strong-maximality-product-exponential} and
Example~\ref{not-not-total-maximal}.

\section{Preliminaries}\label{sec:preliminaries}
We give the basic definitions and results in the theory of (continuous)
dcpos. It is not adequate to simply refer the reader to classical texts on
domain theory~\citep{AbramskyJung1995,GierzEtAl2003}, because two classically
equivalent definitions need not be constructively equivalent, and hence we need
to make choices here. For example, while classically every Scott open subset is
the complement of a Scott closed subset, this does not hold constructively
(Lemma~\ref{complement-of-open-is-closed}).

The results presented here are all provable constructively. Constructive proofs
of standard order-theoretic results, such as Lemma~\ref{way-below-basics} and
Proposition~\ref{interpolation} (the interpolation property), can be found
in~\citep{deJongEscardo2021a}, or alternatively, the author's PhD
thesis~\citep{deJong2022}, and will not be repeated here.
Since the Scott topology is not treated in the above references, the proofs of
any results involving topology (such as Proposition~\ref{basis-is-dense}) are
spelled out in full.

Finally, in Section~\ref{sec:decidability-conditions} we introduce and study
some decidability conditions on bases of dcpos that will make several
appearances throughout the paper. These decidability conditions always
hold if excluded middle is assumed.

\begin{definition}[Directed complete poset (dcpo)]\hfill
  \begin{enumerate}
  \item A subset \(S\) of a poset \((X,\below)\) is \emph{directed} if it is
    \emph{inhabited} (meaning there exists \(s \in S\)) and \emph{semidirected}:
    for every two points \(x,y \in S\) there exists \(z \in S\) with
    \(x \below z\) and \(y \below z\).
  \item A \emph{directed complete poset (dcpo)} is a poset where every
    directed subset \(S\) has a supremum, denoted by~\(\dirsup S\).
  \item A dcpo is \emph{pointed} if it has a least element, typically denoted by
    \(\bot\).
  \end{enumerate}
\end{definition}

Notice that a poset is a pointed dcpo if and only if it has suprema for all
semidirected subsets.
In fact, given a pointed dcpo \(D\) and a semidirected subset \(S \subseteq D\),
we can consider the directed subset \(S \cup \set{\bot}\) of \(D\) whose
supremum is also the supremum of \(S\).

\begin{definition}[Way-below relation \({\ll}\)]
  An element \(x\) of a dcpo \(D\) is \emph{way below} an element \(y \in D\),
  denoted \(x \ll y\), if for every directed subset \(S\) with
  \(y \below \dirsup S\) there exists \(s \in S\) such that \(x \below s\)
  already.
\end{definition}
\begin{lemma}\label{way-below-basics}
  The way-below relation enjoys the following properties:
  \begin{enumerate}
  \item it is transitive;
  \item if \(x \below y \ll z\), then \(x \ll z\) for every \(x\), \(y\) and
    \(z\);
  \item if \(x \ll y \below z\), then \(x \ll z\) for every \(x\), \(y\) and
    \(z\).
  \end{enumerate}
\end{lemma}
\begin{definition}[Continuity of a dcpo]
  A dcpo \(D\) is \emph{continuous} if for every element \(x \in D\),
  the subset \(\ddset x \coloneqq \set*{y \in D \mid y \ll x}\) is directed and
  its supremum is \(x\).
\end{definition}
\begin{definition}[Compactness and algebraicity]
  An element of a dcpo is \emph{compact} if it is way below itself.  A dcpo
  \(D\) is \emph{algebraic} if for every element \(x \in D\), the
  subset \(\set{c \in D \mid c \below x \text{ and \(c\) is
    compact}}\) is directed with supremum \(x\).
\end{definition}
\begin{proposition}
  Every algebraic dcpo is continuous.
\end{proposition}

\subsection{The Scott topology}
We stress that constructively it is necessary to define Scott closed and Scott
open subsets independently, see Lemma~\ref{complement-of-open-is-closed}.

\begin{definition}[Scott topology]\hfill
  \begin{enumerate}
  \item A subset \(C\) of a dcpo \(D\) is \emph{Scott closed} if it is
    closed under directed suprema and a lower set: if \(x \below y \in C\), then
    \(x \in C\) too.
  \item A subset \(U\) of a dcpo \(D\) is \emph{Scott open} if it is an
    upper set and for every directed subset \(S \subseteq D\) with
    \(\dirsup S \in U\), there exists \(s \in S\) such that \(s \in U\) already.
  \end{enumerate}
\end{definition}
\begin{example}\label{ex:Scott-opens}
  For any element \(x\) of a dcpo \(D\), the subset
  \(\dset x \coloneqq \set{y \in D \mid y \below x}\) is Scott closed.
  If \(D\) is continuous, then the subset
  \(\upupset x \coloneqq \set{y \in D \mid x \ll y}\) is Scott open.
  (One proves this using the interpolation property, which is
  Proposition~\ref{interpolation} below.)
  Moreover, if \(D\) is continuous, then the set
  \(\set{\upupset x \mid x \in X}\) is a basis for the Scott topology on \(D\).
\end{example}
\begin{lemma}\label{complement-of-open-is-closed}
  The complement of a Scott open subset is Scott closed.
  The converse holds if and only if excluded middle does, as we prove in
  Proposition~\ref{complement-of-closed-is-open-em}.
\end{lemma}
\begin{proof}
  Let \(U \subseteq D\) be a Scott open of a dcpo \(D\). We show that
  \(C \coloneqq D \setminus U\) is Scott closed.
  If \(x \below y\) with \(y \in C\), then \(x \in C\), for assuming \(x \in U\)
  leads to \(y \in U\) as \(U\) is Scott open, contradicting that \(y \in C\).
  Now suppose that \(S \subseteq C\) is directed and assume for a contradiction
  that \(\bigsqcup S \in U\). By Scott openness of \(U\), there exists
  \(s \in S\) such that \(s \in U\), but this contradicts \(S \subseteq C\).
\end{proof}

\begin{definition}[Interior and closure]\label{def:interior-closure}
  In a topological space \(X\), the \emph{interior} of a subset
  \(S \subseteq X\) is the largest open of \(X\) contained in \(S\), i.e.\ the
  interior of \(S\) is
  \(\bigcup\set{U \in \powerset{X} \mid U \subseteq S, U \text{ is open}}\).
  Dually, the \emph{closure} of a subset \(S \subseteq X\) is the smallest
  closed subset of \(X\) that contains~\(S\).
\end{definition}

\subsection{(Abstract) bases}
Ideal completions of abstract bases provide a source of continuous (and
algebraic) dcpos. In general, the notion of a basis is constructively quite
interesting, as explained in the next subsection.

\begin{definition}[Basis for a dcpo]
  A \emph{basis} for a dcpo \(D\) is a subset \(B \subseteq D\) such that for
  every element \(x \in D\), the subset \({B \cap \ddset{x}}\) is directed with
  supremum \(x\).
\end{definition}

\begin{lemma}
  A dcpo is continuous if and only if it has a basis and a dcpo is algebraic if
  and only if it has a basis of compact elements.
  Moreover, if \(B\) is a basis for an algebraic dcpo \(D\), then \(B\) must
  contain every compact element of \(D\). Hence, an algebraic dcpo has a unique
  smallest basis consisting of compact elements.
\end{lemma}

\begin{example}[Kuratowski finite subsets]\label{ex:powerset}
  The powerset \(\powerset{X}\) of any set \(X\) ordered by inclusion and with
  suprema given by unions is a pointed algebraic dcpo. Its compact elements are
  the Kuratowski finite subsets of \(X\).
  A set \(X\) is \emph{Kuratowski finite} if it is finitely enumerable, that is,
  there exists a surjection \(\set{0,\dots,n-1} \twoheadrightarrow X\) for some
  number \(n \in \Nat\).
\end{example}

\begin{definition}[Sierpi\'nski domain \(\Sierp\)]\label{def:Sierpinski-domain}
  The \emph{Sierpi\'nski domain} \(\Sierp\) is the free pointed dcpo on a single
  generator.
  We can realize \(\Sierp\) as the set of truth values, i.e.\ as the powerset
  \(\powerset{\set{*}}\) of a singleton. The compact elements of \(\Sierp\) are
  exactly the elements \(\bot \coloneqq \emptyset\) and
  \(\top \coloneqq \set{*}\).
\end{definition}

\begin{lemma}\label{basis-basis}
  For every continuous dcpo \(D\), if a subset \(B \subseteq D\) is a basis for
  the dcpo \(D\), then \(\set{\upupset b \mid b \in B}\) is a basis for the
  Scott topology on \(D\).
\end{lemma}
\begin{proof}
  Suppose that \(B\) is a basis of the continuous dcpo \(D\) and let
  \(U \subseteq D\) be an arbitrary Scott open.
  We claim that \(U = \bigcup\set{\upupset b \mid b \in B \cap U}\) which would
  finish the proof.
  If \(b \in B \cap U\), then \(\upupset b \subseteq U\) because \(U\) is upper
  closed.
  Conversely, if \(x \in U\), then, since \(B\) is a basis, \(x\) is the
  directed supremum \(\bigsqcup \set{b \in B \mid b \ll x}\).
  Because \(U\) is Scott open, there must exist \(b \in B\) with \(b \ll x\)
  such that \(b \in U\). Hence, \(b \in B \cap U\) and \(x \in \upupset b\), as
  desired.
\end{proof}

\begin{proposition}\label{basis-is-dense}
  Every basis of a continuous dcpo is dense with respect to the Scott topology
  in the following (classically equivalent) ways:
  \begin{enumerate}
  \item the Scott closure of the basis is the whole dcpo;
  \item every inhabited Scott open contains a point in the basis.
  \end{enumerate}
\end{proposition}
\begin{proof}
  Suppose that \(B\) is a basis of a continuous dcpo \(D\). (1): Let \(C\) be an
  arbitrary Scott closed subset of \(D\) containing \(B\). We wish to show that
  \(D \subseteq C\), so let \(x \in D\) be arbitrary. Since \(B\) is a basis, we
  have \(x = \dirsup B \cap \ddset x\). Note that
  \(B \cap \ddset x \subseteq B \subseteq C\), so \(C\) must also contain~\(x\)
  because Scott closed subsets are closed under directed suprema.
  (2): If \(U\) is a Scott open containing a point \(x\), then from
  \(x = \dirsup B \cap \ddset x\) and the fact that \(U\) is Scott open, we see
  that \(U\) must contain an element of~\(B\).
\end{proof}
\begin{proposition}[Interpolation]\label{interpolation}
  If \(x \ll y\) are elements of a continuous dcpo \(D\), then there exists
  \(b \in D\) with \(x \ll b \ll y\).
  Moreover, if \(D\) has a basis \(B\), then there exists such an element \(b\)
  in \(B\).
\end{proposition}
\begin{lemma}\label{below-in-terms-of-way-below}
  For every two elements \(x\) and \(y\) of a continuous dcpo \(D\) we have
  \[
    x \below y
    \iff \forall_{z \in D}\,\pa*{z \ll x \to z \ll y}
    \iff \forall_{z \in D}\,\pa*{z \ll x \to z \below y}.
  \]
  Moreover, if \(D\) has a basis \(B\), then
  \[
    x \below y
    \iff \forall_{b \in B}\,\pa*{b \ll x \to b \ll y}
    \iff \forall_{b \in B}\,\pa*{b \ll x \to b \below y}.
  \]
\end{lemma}
\begin{definition}[Abstract basis and ideal completion \(\Idl(B,\prec)\)]
  An \emph{abstract basis} is a pair \((B,\prec)\) such that \(\prec\) is
  transitive and \emph{interpolative}: for every \(b \in B\), the subset
  \(\dset b \coloneqq \set{a \in B \mid a \prec b}\) is directed.
  The \emph{rounded ideal completion} \(\Idl(B,\prec)\) of an abstract basis
  \((B,\prec)\) consists of directed lower sets of \((B,\prec)\), known as
  \emph{(rounded) ideals}, ordered by subset inclusion. It is a continuous dcpo
  with basis \(\set{\dset b \mid b \in B}\) and directed suprema given by unions.
\end{definition}
\begin{lemma}\label{reflexive-algebraic}
  If the relation \(\prec\) of an abstract basis \((B,\prec)\) is reflexive,
  then \(\Idl(B,\prec)\) is algebraic and its compact elements are exactly those
  of the form \(\dset b\) for \(b \in B\).
\end{lemma}
\subsection{Decidability conditions}\label{sec:decidability-conditions}
Every continuous dcpo \(D\) has a basis, namely \(D\) itself. Our interest in
bases lies in the fact that we can ask a dcpo to have a basis satisfying certain
decidability conditions that we couldn't reasonably impose on the entire
dcpo. For instance, the basis \(\set{\bot,\top}\) of the Sierpi\'nski domain
\(\Sierp\) has decidable equality, but decidable equality on all of \(\Sierp\)
is equivalent to excluded middle.

The first decidability condition that we will consider is for bases \(B\) of a
pointed continuous dcpo:
\begin{equation}\label{equals-bot-decidable}\tag{\(\delta_\bot\)}
  \text{For every } b \in B, \text{ it is decidable whether } b = \bot.
\end{equation}
The second and third decidability conditions are for bases of any continuous dcpo:
\begin{gather}\label{way-below-decidable}\tag{\(\delta_{\ll}\)}
  \text{For every } a,b \in B, \text{ it is decidable whether } a \ll b. \\
  \label{below-decidable}\tag{\(\delta_{\below}\)}
  \text{For every } a,b \in B, \text{ it is decidable whether } a \below b.
\end{gather}

Observe that each of \eqref{below-decidable} and \eqref{way-below-decidable}
implies \eqref{equals-bot-decidable} for pointed dcpos, because \(\bot\) is
compact.
In general, neither of the conditions \eqref{below-decidable} and
\eqref{way-below-decidable} implies the other, at least as far as we
know. However, in some cases, for example when \(B\) is Kuratowski finite,
\eqref{way-below-decidable} is a stronger condition, because of
Lemma~\ref{below-in-terms-of-way-below}.
Moreover, if the dcpo is algebraic then conditions \eqref{below-decidable} and
\eqref{way-below-decidable} are equivalent for the unique basis of compact
elements.

Finally, we remark that many natural examples in domain theory satisfy the
decidability conditions. In particular, this holds for all examples in
Section~\ref{sec:examples}.

Again, we stress that these decidability conditions are necessarily restricted
to the basis, as the following proposition shows.

\begin{proposition}
  Let \(D\) be any pointed dcpo that is nontrivial in the sense that there
  exists \(x \in D\) with \(x \neq \bot\). If \(y = \bot\) is decidable for
  every \(y \in D\), then weak excluded middle follows.
\end{proposition}
\begin{proof}
  For any proposition \(P\), observe that \(\lnot P\) holds if and only if
  \(\dirsup \set{x \mid P} = \bot\). Hence, if \(y = \bot\) is decidable for
  every \(y \in D\), then so is \(\lnot P\) for every proposition \(P\).
\end{proof}
Moreover, if the order relation of a dcpo is decidable, then, by antisymmetry,
the dcpo must have decidable equality, but we showed in
\cite[Corollary~39]{deJongEscardo2021b} that this implies (weak) excluded
middle, unless the dcpo is trivial.

For understanding these results, it is important to recall that the two-element
poset with \(\bot \below \top\) cannot be shown to be directed complete,
constructively.

We now show that the decidability conditions \eqref{equals-bot-decidable},
\eqref{way-below-decidable} and \eqref{below-decidable} are preserved by taking
products and exponentials (function spaces) of dcpos.

\begin{definition}[Product of dcpos \(D \times E\)]
  The \emph{product} \(D \times E\) of two dcpos \(D\) and \(E\) is given by
  their Cartesian product ordered pairwise. The supremum of a
  directed subset \(S \subseteq D \times E\) is given by the pair of suprema
  \(\dirsup \set{x \in D \mid \exists_{y \in E}\, (x,y) \in S}\) and
  \(\dirsup \set{y \in D \mid \exists_{x \in D}\, (x,y) \in S}\).
\end{definition}
\begin{proposition}
  If \(D\) and \(E\) are continuous dcpos with bases \(B_D\) and \(B_E\), then
  \(B_D \times B_E\) is a basis for the product \(D \times E\). Also, if
  \(B_D\) and \(B_E\) both satisfy \eqref{equals-bot-decidable}, then so does
  \(B_D \times B_E\), and similarly for \eqref{way-below-decidable} and
  \eqref{below-decidable}.
\end{proposition}

\begin{definition}[Scott continuity]
  A function between dcpos is \emph{Scott continuous} if it preserves directed
  suprema.
\end{definition}
\begin{definition}[Exponential of dcpos \(E^D\)]
  The \emph{exponential} \(E^D\) of two dcpos \(D\) and \(E\) is given by the
  set of Scott continuous functions from \(D\) to \(E\) ordered pointwise, i.e.\
  \(f \below g\) if \(\forall_{x \in D}\,f(x) \below g(x)\) for
  \(f,g\colon D \to E\). Suprema of directed subsets are also given pointwise.
\end{definition}

We use the name ``exponential'' for \(E^D\), because this construction (together
with the product) witnesses that the category of dcpos is cartesian closed.
In the (bounded complete) algebraic case, a basis for the exponential can be
constructed using step functions which we recall now.

\begin{definition}[Step function]
  Given an element \(x\) of a dcpo \(D\) and an element \(y\) of a pointed dcpo
  \(E\), the \emph{single-step function}
  \(\steppa{x \To y} \colon D \to E\) is defined as
  \(\steppa{x \To y}(d) \coloneqq \dirsup \set{y \mid x \below d}\).
  A \emph{step-function} is the supremum of a Kuratowski finite
  (recall~Example~\ref{ex:powerset}) subset of single-step functions.
\end{definition}

\begin{lemma}\label{step-function-compact}
  If \(x\) is a compact element of \(D\) and \(y\) is any element of a pointed
  dcpo \(E\), then the single-step function
  \(\steppa{x \To y} \colon {D \to E}\) is Scott continuous. If \(y\) is also
  compact, then \(\steppa{x \To y}\) is a compact element of the exponential
  \(E^D\).
\end{lemma}

\begin{lemma}\label{step-function-below}
  For every element \(x\) of a dcpo \(D\), element \(y\) of a pointed dcpo \(E\)
  and Scott continuous function \(f \colon D \to E\), we have
  \(\steppa{x \To y} \below f\) if and only if \(y \below f(x)\).
\end{lemma}

\begin{lemma}\label{compact-closed-under-finite-sups}
  The compact elements of a dcpo are closed under existing Kuratowski finite
  suprema.
\end{lemma}

\begin{definition}[Bounded completeness]
  A subset \(S\) of a poset \((X,\below)\) is \emph{bounded} if there exists
  \(x \in X\) such that \(s \below x\) for every \(s \in S\).
  A poset \((X,\below)\) is \emph{bounded complete} if every bounded
  subset \(S \subseteq X\) has a supremum \(\dirsup S\) in \(X\).
\end{definition}

\begin{proposition}\label{dec-closed-under-exponentials}
  If \(D\) is an inhabited algebraic dcpo with basis of compact elements \(B_D\)
  and \(E\) is a pointed bounded complete algebraic dcpo with basis of compact
  elements \(B_E\), then
  \begin{align*}
    B \coloneqq \Big{\{}\dirsup S \mid\,\, &S \text{ is a bounded Kuratowski finite
    subset of single-step functions} \\
    &\text{of the form } \steppa{a \To b} \text { with } a \in B_D
      \text{ and } b \in B_E\Big{\}}.
  \end{align*}
  is the basis of compact elements for the algebraic exponential \(E^D\).
  Moreover, if \(B_E\) satisfies \eqref{equals-bot-decidable}, then so
  does~\(B\). Finally, if \(B_D\) and \(B_E\) both satisfy
  \eqref{below-decidable} \textup{(}or equivalently,
  \eqref{way-below-decidable}\textup{)}, then \(B\) satisfies
  \eqref{below-decidable} and \eqref{way-below-decidable} too.
\end{proposition}
\begin{proof}
  Firstly, notice that \(B\) is well-defined as suprema are given pointwise and
  \(E\) is bounded complete. Secondly, \(B\) consists of compact elements by
  Lemmas~\ref{step-function-compact}~and~\ref{compact-closed-under-finite-sups}.
  Next, we show that \(B\) is indeed a basis for~\(E^D\). So suppose that
  \(f \colon D \to E\) is Scott continuous. We show that
  \(B \cap \ddset f = B \cap \dset f\) is directed. Observe that this set is
  indeed semidirected: if \(S\) and \(T\) are Kuratowski finite subsets of
  single-step functions bounded by \(f\), then so is \(S \cup T\) and its
  supremum is above those of \(S\) and \(T\). Moreover, the set is inhabited,
  because \(D\) is inhabited, so there must exist some \(b \in B_D\) and then
  \(\steppa{b \To \bot}\) is a single-step function way below \(f\). Hence,
  \(B \cap \dset f\) is directed, as desired.
  Next, we prove that \(f\) equals the supremum of \(B \cap \dset f\), which
  exists because \(B \cap \dset f\) is bounded by \(f\).
  So let \(x \in D\) be arbitrary. We must show that
  \(f(x) \below \pa*{\dirsup\pa*{B \cap \dset f}}(x)\). Using continuity of
  \(f\), it suffices to prove that
  \(f(a) \below \pa*{\dirsup\pa*{B \cap \dset f}}(x)\) for every \(a \in B_D\)
  with \(a \below x\). So fix \(a \in B_D\) with \(a \below x\). Since \(B_E\)
  is a basis for \(E\), it is enough to prove that
  \(b \below \pa*{\dirsup\pa*{B \cap \dset f}}(x)\) for every \(b \in B_E\) with
  \(b \below f(a)\). But if we have \(b \in B_E\) with \(b \below f(a)\), then
  \(\steppa{a \To b} \in B \cap \dset f\) by Lemma~\ref{step-function-below}, so
  \(b = \steppa{a \To b}(x) \below \pa*{\dirsup\pa*{B \cap \dset f}}(x)\), as
  desired.

  Now suppose that \(B_E\) satisfies \eqref{equals-bot-decidable} and let
  \(\dirsup S \in B\). Notice that the least element of \(E^D\) is given by
  \(x \mapsto \bot\). We wish to show that it is decidable whether
  \(\dirsup S = \pa*{x \mapsto \bot}\). Since \(S\) is Kuratowski finite,
  it is enough to show that \(\steppa{a \To b} = \pa*{x \mapsto \bot}\) is
  decidable for every \(\steppa{a \To b} \in S\). By
  Lemma~\ref{step-function-below}, this inequality reduces to \(b = \bot\),
  which is indeed decidable as \(B_E\) is assumed to satisfy
  \eqref{equals-bot-decidable}.

  Finally, assume that both \(B_D\) and \(B_E\) satisfy \eqref{below-decidable}
  and suppose that \(\dirsup S, \dirsup T \in B\). We must prove that
  \(\dirsup S \below \dirsup T\) is decidable. Since \(S\) is Kuratowski finite,
  it is enough to prove that \(\steppa{a \To b} \below \dirsup T\) is decidable
  for every \(\steppa{a \To b} \in S\). By Lemma~\ref{step-function-below}, this
  inequality reduces to \(b \below \pa*{\dirsup T}(a)\). Write \(T\) as a
  Kuratowski finite set of single-step functions:
  \(T = \set{\steppa{a_1 \To b_1},\steppa{a_2 \To b_2},\dots,\steppa{a_n \To
      b_n}}\). Then
  \(\pa*{\dirsup T}(a) = \dirsup\set{\chi_1(a),\chi_2(a),\dots,\chi_n(a)}\)
  where
  \[
    \chi_i(a) \coloneqq
    \begin{cases}
      b_i &\text{if } a_i \below a; \\
      \bot &\text{else};
    \end{cases}
  \]
  using that \(B_D\) satisfies \eqref{below-decidable}. Hence, since \(B_E\) is
  closed under finite bounded suprema by
  Lemma~\ref{compact-closed-under-finite-sups}, we see that
  \(\pa*{\dirsup T}(a) \in B_E\). But then \(b \below \pa*{\dirsup T}(a)\) is
  decidable because \(B_E\) satisfies \eqref{below-decidable} by assumption.
\end{proof}

\section{The intrinsic apartness}\label{sec:intrinsic-apartness}
\begin{definition}[Specialization preorder \(\leq\) and \(\posnotbelow\)]\label{def:specialization}
  The \emph{specialization preorder} on a topological space \(X\) is the
  preorder \({\leq}\) on \(X\) given by putting \(x \leq y\) if every open
  neighbourhood of \(x\) is an open neighbourhood of \(y\). Given \(x,y\in X\),
  we write \(x \posnotbelow y\) if there exists an open neighbourhood of \(x\)
  that does not contain \(y\).
\end{definition}

Observe that \(x \posnotbelow y\) is classically equivalent to \(x \nleq y\),
the logical negation of \(x \leq y\). We also write \(x \not\below y\) for the
logical negation of \(x \below y\), where \({\below}\) is the partial order on a
dcpo.

\begin{definition}[Apartness]\label{def:apartness}
  An \emph{apartness} on a set \(X\) is a binary relation \({\apart}\) on \(X\)
  satisfying
  \begin{enumerate}
  \item \emph{irreflexivity}: \(x \apart x\) is false for every \(x \in X\);
  \item \emph{symmetry}: if \(x \apart y\), then \(y \apart x\) for every
    \(x,y \in X\).
  \end{enumerate}
  If \(x \apart y\) holds, then \(x\) and \(y\) are said to be \emph{apart}.
\end{definition}
Notice that we do not require cotransitivity or tightness, cf.\
Warning~\ref{warning} on~page~\pageref{warning}. Also note that irreflexivity
implies that if \(x \apart y\), then \(x \neq y\), so \(\apart\) is a
strengthening of inequality.

\begin{definition}[Intrinsic apartness \({\apart}\)]\label{def:intrinsic-apartness}
  Given two points \(x\) and \(y\) of a topological space \(X\), we say that
  \(x\) and \(y\) are \emph{intrinsically apart}, written \(x \apart y\), if
  \(x \posnotbelow y\) or
  \(y \posnotbelow x\).
  Thus, \(x\) is intrinsically apart from \(y\) if there is a open neighbourhood
  of \(x\) that does not contain \(y\) or vice versa.
  It is clear that the relation \({\apart}\) is an apartness in the sense of
  Definition~\ref{def:apartness}.
\end{definition}

With excluded middle, one can show that the specialization preorder for the
Scott topology on a dcpo coincides with the partial order of the dcpo. In
particular, the specialization preorder is in fact a partial
order. Constructively, we still have the following result.

\begin{lemma}\label{order-specialization}
  If \(x \below y\) in a dcpo \(D\), then \(x \leq y\), where \({\leq}\) is the
  specialization order of the Scott topology on \(D\). If \(D\) is continuous,
  then the converse holds too, so \({\below}\) and \({\leq}\) coincide in that
  case.
\end{lemma}
\begin{proof}
  The first claim follows because Scott opens are upper sets. Now assume that
  \(D\) is continuous and that \(x \leq y\). Then
  \(\forall_{z \in X}\,\pa*{x\in\upupset z \to y \in\upupset z}\) by
  Example~\ref{ex:Scott-opens} and hence, \(x \below y\) by
  Lemma~\ref{below-in-terms-of-way-below}.
\end{proof}

\begin{lemma}\label{posnotbelow-criterion}
  For a continuous dcpo \(D\) we have \(x \posnotbelow y\) if and only if there
  exists \(b \in D\) such that \(b \ll x\), but \(b \not\below y\).
  Moreover, if \(D\) has a basis \(B\), then there exists such an element \(b\)
  in \(B\).
\end{lemma}
\begin{proof}
  Suppose that \(B\) (which may be all of \(D\)) is a basis for \(D\). If we
  have \(b \in B\) with \(b \ll x\) and \(b \not\below y\), then \(\upupset b\)
  is a Scott open containing \(x\), but not \(y\), so \(x \posnotbelow y\).
  Conversely, if there exists a Scott open \(U\) with \(x \in U\) and
  \(y \not\in U\), then by Lemma~\ref{basis-basis}, there exists \(b \in B\)
  such that \(x \in \upupset b \subseteq U\). Hence, \(b \ll x\), but
  \(b \not\below y\), for if \(b \below y\), then \(y \in U\) as \(U\) is an
  upper set, contradicting our assumption.
\end{proof}
The condition in Lemma~\ref{posnotbelow-criterion} appears in a remark right
after \cite[Definition~I-1.6]{GierzEtAl2003}, as a classically equivalent
reading of \(x \not\below y\).

\begin{example}
  Consider the powerset \(\powerset{X}\) of a set \(X\) as a pointed algebraic
  dcpo. Using Lemma~\ref{posnotbelow-criterion}, we see that a subset
  \(A \in \powerset{X}\) is intrinsically apart from the empty set if and only
  if \(A\) is inhabited. More generally, for \(A,B \in \powerset{X}\), we have
  \(A \posnotbelow B\) if and only if \(B \setminus A\) is inhabited.
\end{example}

\begin{proposition}\label{posnotbelow-notbelow}\hfill
  \begin{enumerate}
  \item\label{posnotbelow-implies-notbelow} For any elements \(x\) and \(y\) of
    a dcpo \(D\), we have that \(x \posnotbelow y\) implies \(x \not\below y\).
  \item The converse of~\eqref{posnotbelow-implies-notbelow} holds if and only
    if excluded middle holds.
    In particular, if the converse of~\eqref{posnotbelow-implies-notbelow} holds
    for all elements of the Sierpi\'nski domain \(\Sierp\), then excluded middle
    follows.
  \item\label{apart-implies-neq} For any elements \(x\) and \(y\) of a dcpo \(D\), we have that
    \(x \apart y\) implies \(x \neq y\).
  \item The converse of~\eqref{apart-implies-neq} holds if and only if excluded
    middle holds.
    In particular, if the converse of~\eqref{apart-implies-neq} holds for all
    elements of the Sierpi\'nski domain \(\Sierp\), then excluded middle
    follows.
  \item If \(c\) is a compact element of a dcpo \(D\) and \(x \in D\), then
    \(c \not\below x\) implies \(c \posnotbelow x\), without the need to assume
    excluded middle.
  \end{enumerate}
\end{proposition}
\begin{proof}
  (1): This is just the contrapositive of the first claim in
  Lemma~\ref{order-specialization}.
  (2): It is straightforward to prove the implication if excluded middle holds.
  For the converse, assume that \(x \not\below y\) implies \(x \posnotbelow y\)
  for every \(x,y \in \Sierp\). We are going to show that \(\lnot\lnot P \to P\)
  for every proposition \(P\), which is equivalent to excluded middle. So let
  \(P\) be an arbitrary proposition and assume that \(\lnot\lnot P\) holds. Then
  the element \(\chi_P \coloneqq \set{* \mid P} \in \Sierp\) satisfies
  \(\chi_P \not\below \bot\). So \(\chi_P \posnotbelow \bot\) by assumption. By
  Lemma~\ref{posnotbelow-criterion} and Example~\ref{ex:powerset}, there exists
  \(b \in \set{\bot,\top}\) such that \(b \ll \chi_P\) and \(b \neq \bot\). The
  latter implies \(b = \top = \set{*}\), which means that \(\chi_P\) must be
  inhabited, which is equivalent to \(P\) holding.
  (3): Since the intrinsic apartness is irreflexive.
  (4): Similar to~(2).
  (5): If \(c\) is compact and \(c \not\below x\), then \(\upupset c\) is a
  Scott open containing \(c\), but not \(x\).
\end{proof}

With excluded middle, complements of Scott closed subsets are Scott open. In
particular, the subset \(\set{x \in D \mid x \not\below y}\) is Scott open for
any element \(y\) of a dcpo \(D\). Constructively, we have the following result:
\begin{lemma}\label{interior-of-complement}
  For any element \(y\) of a dcpo \(D\), the Scott interior of
  \(\set{x \in D \mid x \not\below y}\) is given by the subset
  \(\set{x \in D \mid x \posnotbelow y}\), where we recall that
  \(x \posnotbelow y\) means that there exists a Scott open containing~\(x\) but
  not \(y\).
\end{lemma}
\begin{proof}
  Notice that \(\set{x \in D \mid x \posnotbelow y}\) is the same thing as
  \(\bigcup\set{U \subseteq D \mid U \text{ is Scott open and } y \not\in U}\),
  which is a union of opens and therefore open itself. Moreover,
  \(\set{x \in D \mid x \posnotbelow y}\) is contained in
  \(\set{x \in D \mid x \not\below y}\) by
  Proposition~\ref{posnotbelow-notbelow}.
  Finally, if \(V\) is a Scott open contained in
  \(\set{x \in D \mid x \not\below y}\), then \(y \not\in V\), since
  \(y \not\below y\) is false. Hence,
  \(V \subseteq \bigcup\set{U \subseteq D\mid U \text{ is Scott open and } y
    \not\in U} = \set{x \in D \mid x \posnotbelow y}\), completing the proof.
\end{proof}

\begin{proposition}\label{complement-of-closed-is-open-em}
  If the complement of every Scott closed subset of the Sierpi\'nski domain
  \(\Sierp\) is Scott open, then excluded middle follows.
\end{proposition}
\begin{proof}
  If the complement of every Scott closed subset of \(\Sierp\) is Scott open,
  then \(x \not\below y\) implies \(x \posnotbelow y\) for every two points
  \(x,y \in \Sierp\) by Lemma~\ref{interior-of-complement}. Hence,
  excluded middle would follow by Proposition~\ref{posnotbelow-notbelow}.
\end{proof}

An element \(x\) of a pointed dcpo may be said to be nontrivial if
\(x \neq \bot\). Given our notion of apartness, we might consider the
constructively stronger \(x \apart \bot\). We show that this is related to
Johnstone's notion of positivity~\cite[p.~98]{Johnstone1984}.
In~\cite[Definition~25]{deJongEscardo2021b} we adapted Johnstone's notion of
positivity from locales to a general class of posets that includes dcpos. Here
we give an equivalent, but simpler, definition just for pointed dcpos.
\begin{definition}[Positivity]
  An element \(x\) of a pointed dcpo \(D\) is \emph{positive} if every
  semidirected subset \(S \subseteq D\) satisfying \(x \below \dirsup S\) is
  inhabited (and hence directed).
\end{definition}

\begin{proposition}
  For every element \(x\) of a pointed dcpo, if \(x \apart \bot\), then \(x\) is
  positive.
  In the other direction, if \(D\) is a continuous pointed dcpo with a basis
  satisfying \eqref{equals-bot-decidable}, then every positive element of \(D\)
  is apart from~\(\bot\).
\end{proposition}
\begin{proof}
  Suppose first that \(x\) is an element of a pointed dcpo \(D\) and that
  \(x \apart \bot\). Then we have a Scott open~\(U\) containing \(x\) but not
  \(\bot\). Now suppose that \(S \subseteq D\) is semidirected with
  \(x \below \dirsup S\). Scott opens are upper sets, so \(\dirsup S \in U\).
  Now \(S' \coloneqq S \cup \set{\bot}\) is a directed subset of \(D\) with the
  same supremum as \(S\). In particular, \(\dirsup S' \in U\). But then there
  exists \(y \in S'\) such that \(y \in U\) already, because \(U\) is Scott
  open.
  By construction of \(S'\), we have \(y = \bot\) or \(y \in S\).
  As the former is impossible, we are done.

  For the other claim, assume that \(D\) is continuous with a basis \(B\)
  satisfying \eqref{equals-bot-decidable} and that \(x \in D\) is
  positive. Consider the subset
  \(S \coloneqq \set{b \in B \mid b \ll x , b \neq \bot}\) and note that it is
  semidirected because \(B \cap \ddset x\) is. If we can show that \(S\) is
  inhabited, then we see that \(x \apart \bot\) by
  Lemma~\ref{posnotbelow-criterion}, which would finish the proof. As \(x\) is
  positive, it suffices to prove that \(x \below \dirsup S\). Since \(B\) is a
  basis, it is enough to prove that \(b \below \dirsup S\) for every \(b \in B\)
  with \(b \ll x\). So suppose that we have \(b \in B\) with \(b \ll x\). By the
  decidability condition \eqref{equals-bot-decidable}, either \(b = \bot\) or
  \(b \neq \bot\). In the first case, we get \(b = \bot \below \dirsup S\), as
  desired; and in the second case, we get \(b \below \dirsup S\) because
  \(b \in S\).
\end{proof}

\section{The apartness topology}\label{sec:apartness-topology}
In~\cite[Section~2.2]{BridgesVita2011}, the authors start with a
topological space \(X\) equipped with an apartness relation~\({\apart}\) and,
using the topology and apartness, define a second topology on~\(X\), known as
the \emph{apartness topology}. A~natural question is whether the original
topology and the apartness topology coincide. For example, if \(X\) is a metric
space and we set two points of \(X\) to be apart if their distance is strictly
positive, then the Bridges-V\^i\c{t}\v{a} apartness topology and the topology
induced by the metric coincide (Proposition~2.2.10 in ibid).
We show, assuming a modest \(\lnot\lnot\)-stability condition that holds in
examples of interest, that the Scott topology on a continuous dcpo with the
intrinsic apartness relation coincides with the apartness topology.

We start by repeating some basic definitions and results of Bridges and
V\^i\c{t}\v{a}.  Recalling Warning~\ref{warning} on~page~\pageref{warning}, we
remind the reader familiar with~\citep{BridgesVita2011} that they write \(\neq\)
and use the word \emph{inequality} for what we denote by \(\apart\) and call
\emph{apartness}.

In constructive mathematics, positively defined notions are usually more useful
than negatively defined ones. We already saw examples of this: \({\apart}\)
versus \({\neq}\) and \({\posnotbelow}\) versus \({\not\below}\). We now use an
apartness to give a positive definition of the complement of a set.
\begin{definition}[(Logical) complement; {\cite[pp.~19--20]{BridgesVita2011}}]
  Given a subset \(A\) of a set \(X\) with an apartness \(\apart\) we define the
  \emph{logical complement} and the \emph{complement} respectively as
  \begin{enumerate}
  \item \(\lcompl A \coloneqq \set{x \in X \mid x \not\in A} =
    \set{x \in X \mid \forall_{y \in A}\, x \neq y}\), and
  \item \(\compl A \coloneqq \set{x \in X \mid \forall_{y \in A}\,x \apart y}\).
  \end{enumerate}
\end{definition}

Recall that, classically, a topological space is a \(T_0\) or \emph{Kolmogorov}
space if \(x \neq y\) implies that there exists an open \(U\) such that
\(x \in U\) and \(y \not\in U\).
This explains the name for the following definition of Bridges and
V\^i\c{t}\v{a}.

\begin{definition}[Topological reverse Kolmogorov property; p.~29 in ibid]
  A topological space \(X\) equipped with an apartness \({\apart}\) satisfies
  the \emph{(topological) reverse Kolmogorov property} if for every open \(U\)
  and points \(x,y \in X\) with \(x \in U\) and \(y \not\in U\), we have
  \(x \apart y\).
\end{definition}

\begin{lemma}[Proposition~2.2.2 in ibid]
  \label{reverse-Kolmogorov-consequence}
  If a topological space \(X\) equipped with an apartness \(\apart\) satisfies
  the reverse Kolmogorov property, then for every subset \(A \subseteq X\), we
  have \(\interior{\lcompl A} = \interior{\compl A}\), where \(Y^\circ\)
  denotes the interior of \(Y\) in \(X\).
\end{lemma}

\begin{definition}[Apartness complement and topology, nearly open subset; pp.~20, 28 and 31 in ibid]
  For an element \(x\) of a topological space \(X\) with an apartness \(\apart\)
  and a subset \(A \subseteq X\), we write \(x \setapart A\) if
  \(x \in \interior{\compl A}\). This gives rise to the \emph{apartness
    complement}: \(\acompl A \coloneqq \set{x \in X \mid x \bowtie A}\).
  Subsets of the form \(\acompl A\) are called \emph{nearly open}.
  The \emph{apartness topology} on \(X\) is the topology whose basic opens are
  the nearly open subsets of~\(X\).
\end{definition}

While the name ``nearly open'' may suggest otherwise, the apartness topology is
coarser than the original topology:

\begin{lemma}[Proposition~2.2.7 in ibid]
  \label{nearly-open-is-open}
  Every nearly open subset of \(X\) is open in the original topology of~\(X\).
\end{lemma}

The following are original contributions.
\begin{definition}[\(\lnot\lnot\)-stable basis]
  We say that a basis \(B\) for a topological space \(X\) is
  \emph{\(\lnot\lnot\)-stable} if \(U = \interior{\lnot\lnot U}\) for every open
  \(U \in B\).
  Note that \(U \subseteq \interior{\lnot\lnot U}\) holds for every open \(U\),
  so the relevant condition is that \(\interior{\lnot\lnot U} \subseteq U\) for
  every basic open \(U\).
\end{definition}
Examples of such bases will be provided by Theorem~\ref{apartness-Scott-topology} below.

\begin{lemma}\label{matching-topologies-criterion}
  If a topological space \(X\) equipped with an apartness \(\apart\) satisfies
  the reverse Kolmogorov property and has a \(\lnot\lnot\)-stable basis,
  then the original topology on \(X\) and the apartness topology on \(X\)
  coincide, i.e.\ a subset of \(X\) is open (in the original topology) if and
  only if it is nearly open.
\end{lemma}
\begin{proof}
  By Lemma~\ref{nearly-open-is-open}, every nearly open subset of \(X\) is open
  in the original topology. So it remains to prove the converse. For this it
  suffices to prove that every basic open of \(X\) is nearly open. So let \(U\)
  be an arbitrary basic open of \(X\). We are going to show that
  \(U = \acompl\pa*{\lcompl U}\), which proves the claim, as the latter is
  nearly open. To this end, observe that
  \begin{align*}
    \acompl\pa*{\lcompl U} &= \set{x \in X \mid x \setapart {\lcompl U}} &&\text{(by definition)} \\
    &= \set{x \in X \mid x \in \interior{\compl\pa*{\lcompl U}}} &&\text{(by definition)} \\
    &= \set{x \in X \mid x \in \interior{\lcompl{\lcompl U}}} &&\text{(by Lemma~\ref{reverse-Kolmogorov-consequence})} \\
    &= \set{x \in X \mid x \in U} &&\text{(by \(\lnot\lnot\)-stability of the basis)},
  \end{align*}
  so that \(U = \acompl\pa*{\lcompl U}\), as desired.
\end{proof}

\begin{therm}\label{apartness-Scott-topology}
  Let \(D\) be a continuous dcpo with a basis \(B\). Each of the following
  conditions on the basis \(B\) implies that \(\set{\upupset b \mid b \in B}\) is a
  \(\lnot\lnot\)-stable basis for the Scott topology on \(D\):
  \begin{enumerate}
  \item For every \(a,b\in B\), if \(\lnot\lnot(a \ll b)\), then \(\pa*{a \ll b}\).
  \item For every \(a,b\in B\), if \(\lnot\lnot(a \below b)\), then \(\pa*{a \below b}\).
  \item \(B\) satisfies \eqref{way-below-decidable}
  \item \(B\) satisfies \eqref{below-decidable}
  \end{enumerate}
  Hence, if one of these conditions holds, then the Scott topology on \(D\)
  coincides with the apartness topology on~\(D\) with respect to the intrinsic
  apartness, i.e.\ a subset of \(D\) is Scott open if and only if it is nearly
  open.
\end{therm}

\begin{proof}
  The final claim follows from Lemma~\ref{matching-topologies-criterion} and the
  fact that the intrinsic apartness satisfies the reverse Kolmogorov property
  (by definition). Moreover, (3) implies (1) and (4) implies (2). So it
  suffices to show that if (1) or (2) holds, then
  \(\set{\upupset a \mid a \in B}\) is a \(\lnot\lnot\)-stable basis for the
  Scott topology on \(D\), viz.\ that
  \(\interior{\lnot\lnot\upupset a} \subseteq \upupset a\) for every
  \(a \in B\).  Let \(a \in B\) be arbitrary and suppose that
  \(x \in \interior{\lnot\lnot\upupset a}\).  Using Scott openness and
  continuity of \(D\), there exists \(b \in B\) such that \(b \ll x\) and
  \(b \in \lnot\lnot\upupset a\). The latter just says that
  \(\lnot\lnot (a \ll b)\).
  So if condition~(1) holds, then we get \(a \ll b\), so \(a \ll x\) and
  \(x \in \upupset a\), as desired.
  Now suppose that condition~(2) holds. From \(\lnot\lnot(a \ll b)\), we get
  \(\lnot\lnot(a \below b)\) and hence, \(a \below b\) by condition~(2). So
  \(a \below b \ll x\) and \(x \in \upupset a\), as wished.
\end{proof}

\section{Tightness, cotransitivity and sharpness}\label{sec:tightness-cotransitivity-sharpness}
We show that, in general, the intrinsic apartness will not be tight or
cotransitive. In fact, we cannot expect \emph{any} tight or cotransitive
apartness relation on nontrivial dcpos.
To remedy this, we introduce the notion of sharpness and we show
(Theorem~\ref{sharp-tight-cotransitive}) that the sharp elements do satisfy
tightness and cotransitivity for the intrinsic apartness.

\begin{definition}[Tightness and cotransitivity]
  An apartness relation \(\apartvar\) on a set \(X\) is \emph{tight} if
  \(\lnot(x \apartvar y)\) implies \(x = y\) for every \(x,y \in
  X\). The~apartness is \emph{cotransitive} if \(x\apartvar y\) implies the
  disjunction of \(x \apartvar z\) and \(x \apartvar y\) for every \(x,y,z \in X\).
\end{definition}

\begin{lemma}\label{tight-implies-notnot-sep}
  If \(X\) is a set with a tight apartness, then \(X\) is
  \(\lnot\lnot\)-separated, viz.\ \(\lnot\lnot(x=y)\) implies \(x = y\) for
  every \(x,y \in X\).
\end{lemma}
\begin{proof}
  For any apartness \({\apartvar}\) on \(X\) we have that \(x \apartvar y\)
  implies \(\lnot(x = y)\). Hence, \(\lnot\lnot(x=y)\) implies
  \(\lnot(x\apartvar y)\), so if \({\apartvar}\) is tight, then \(X\) is
  \(\lnot\lnot\)-separated.
\end{proof}

It will be helpful to employ the following positive (but classically equivalent)
formulation of \(\pa*{x \below y} \land \pa*{x \neq y}\) from
\cite[Definition~20]{deJongEscardo2021b}.

\begin{definition}[Strictly below \(\sbelow\)]
  An element \(x\) of a dcpo \(D\) is \emph{strictly below} an element \(y\),
  written \(x \sbelow y\), if \(x \below y\) and for every \(z \aboveorder y\)
  and proposition \(P\), the equality
  \(z = \dirsup\pa*{\set{x} \cup \set{z \mid P}}\) implies \(P\).
\end{definition}

We can relate the above notion to the intrinsic apartness, but, although we do
not have a counterexample, we believe \(x \sbelow y\) to be weaker than
\(x \apart y\) in general. However, if \(x\) is a sharp element of a continuous
dcpo, then \(x \sbelow y\) does imply \(x \apart y\), as shown in
Proposition~\ref{sbelow-implies-apart}.

\begin{proposition}\label{apart-implies-sbelow}
  If \(x \below y\) are elements of a dcpo \(D\) with \(x \apart y\), then
  \(x \sbelow y\).
\end{proposition}
\begin{proof}
  Suppose we have \(x \below y\) and \(x \apart y\). Then \(y \posnotbelow x\)
  and hence by Lemma~\ref{posnotbelow-criterion}, there exists \(b \in D\) with
  \(b \ll y\) but \(b \not\below x\). Now let \(z \aboveorder y\) be arbitrary
  and \(P\) be a proposition with
  \(z = \dirsup\pa*{\set{x} \cup \set{z \mid P}}\). We must prove that \(P\)
  holds. Since \(b \ll y \below z\), we have
  \(b \ll z = \dirsup\pa*{\set{x} \cup \set{z \mid P}}\) and hence
  \(b \below x\) or \(b \below y\) and \(P\) holds. But \(b \not\below x\), so
  \(P\) must hold, as desired.
\end{proof}

\begin{example}
  In the Sierpi\'nski domain \(\Sierp\) we have \(\bot \sbelow \top\). In the
  powerset \(\powerset{X}\) of some set \(X\), the empty set is strictly below a
  subset \(A\) of \(X\) if and only if \(A\) is inhabited. More generally, if
  \(A \subseteq B\) are subsets of some set, then \(A \sbelow B\) holds if
  \(B \setminus A\) is inhabited, and if \(A\) is a decidable subset and
  \(A \sbelow B\), then \(B \setminus A\) is inhabited.
\end{example}

The following shows that we cannot expect \emph{any} tight or cotransitive
apartness relations on nontrivial dcpos.

\begin{therm}\label{tight-cotransitive-wem}
  Let \(D\) be a dcpo with an apartness relation \({\apartvar}\).
  \begin{enumerate}
  \item If \(D\) has elements \(x \sbelow y\), then tightness of \({\apartvar}\)
    implies excluded middle.
  \item If \(D\) has elements \(x \below y\) with \(x \apartvar y\), then
    cotransitivity of \({\apartvar}\) implies weak excluded middle.
  \end{enumerate}
\end{therm}

\begin{proof}
  (1): If \({\apartvar}\) is tight, then \(D\) is \(\lnot\lnot\)-separated by
  Lemma~\ref{tight-implies-notnot-sep}, which, since \(D\) has elements
  \(x \sbelow y\), implies excluded middle by
  \cite[Theorem~38]{deJongEscardo2021b}.
  (2): For a proposition~\(P\), consider the supremum \(s_P\) of the
  directed subset \(\set{x} \cup \set{y \mid P}\). If \({\apartvar}\) is
  cotransitive, then either \(x \apartvar s_P\) or \(y \apartvar s_P\). In the
  first case, \(x \neq s_P\), so that \(\lnot\lnot P\) must be the case. In the
  second case, \(y \neq s_P\), so that \(\lnot P\) must be true. Hence,
  \(\lnot P\) is decidable and weak excluded middle follows.
\end{proof}

Theorem~\ref{tight-cotransitive-wem} works for \emph{any} apartness
relation. The following result specializes to the intrinsic apartness and
derives full excluded middle from cotransitivity.

\begin{therm}\label{tight-cotransitive-em}\hfill
  \begin{enumerate}
  \item If excluded middle holds, then the intrinsic apartness on any dcpo is
    tight.
    In the other direction, if \(D\) is a continuous dcpo with elements
    \(x \below y\) that are intrinsically apart, then tightness of the intrinsic
    apartness on \(D\) implies excluded middle.
  \item If excluded middle holds, then the intrinsic apartness on any dcpo is
    cotransitive.
    In the other direction, if \(D\) is a continuous dcpo and has elements
    \(x \below y\) that are intrinsically apart, then cotransitivity of the
    intrinsic apartness on \(D\) implies excluded middle.
  \item In particular, if the intrinsic apartness on the Sierpi\'nski domain
    \(\Sierp\) is tight or cotransitive, then excluded middle follows.
  \end{enumerate}
\end{therm}
\begin{proof}
  (1): If excluded middle holds, then \(\apart\) coincides with \(\neq\) which
  is easily seen to be tight using excluded middle.
  Now suppose that \(D\) is continuous with elements \(x \below y\) such that
  \(x \apart y\). By Theorem~\ref{tight-cotransitive-wem}, it satisfies to show
  that \(x \sbelow y\), but this follows from
  Proposition~\ref{apart-implies-sbelow}.

  (2): If excluded middle holds, then \(\apart\) coincides with \(\neq\) which
  is easily seen to be cotransitive using excluded middle.
  Now suppose that \(D\) is continuous with elements \(x \below y\) such that
  \(x \apart y\). As in the proof of Theorem~\ref{tight-cotransitive-wem}, for
  any proposition \(P\), consider the supremum \(s_P\) of the directed subset
  \(\set{x} \cup \set{y \mid P}\). If \({\apart}\) is cotransitive, then either
  \(x \apart s_P\) or \(y \apart s_P\).
  If \(y \apart s_P\), then \(\lnot P\) must be the case, as \(P\) implies that
  \(y = s_P\).
  So suppose that \(x \apart s_P\). Then \(s_P \posnotbelow x\), because
  \(x \below s_P\). By Lemma~\ref{posnotbelow-criterion}, there exists
  \(d \in D\) with \(d \ll s_P\) and \(d \not\below x\).  Since \(d \ll s_P\),
  there must exist \(e \in \set{x}\cup\set{y \mid P}\) such that \(d \below
  e\). But \(d \not\below x\), so \(e = y\) and \(P\) must be true. Hence, \(P\)
  is decidable.

  (3): This follows from the above and the fact that the elements
  \(\bot \below \top\) of \(\Sierp\) are intrinsically apart.
\end{proof}

We now isolate a collection of elements, which we call sharp elements, for which
the intrinsic apartness \emph{is} tight and cotransitive. The definition of a
sharp element of a continuous dcpo may be somewhat opaque, but the algebraic
case (Proposition~\ref{sharp-algebraic}) is easier to understand: an element
\(x\) is sharp if and only if \(c \below x\) is decidable for every compact
element \(c\). Sharpness also occurs naturally in our examples in
Section~\ref{sec:examples}, e.g.\ a lower real is sharp if and only if it is
located.

\begin{definition}[Sharpness]
  An element \(x\) of a dcpo \(D\) is \emph{sharp} if for every \(y,z\in D\)
  with \(y \ll z\) we have \(y \ll x\) or \(z \not\below x\).
\end{definition}
Theorem~\ref{basis-is-sharp} below provides many examples of sharp elements.
Our first result is that sharpness in continuous dcpos is equivalent to a
seemingly stronger condition.

\begin{proposition}\label{sharp-seemingly-stronger}
  An element \(x\) of a continuous dcpo \(D\) is sharp if and only if for every
  \(y,z \in D\) with \(y \ll z\) we have \(y \ll x\) or \(z \posnotbelow x\).
\end{proposition}
\begin{proof}
  The right-to-left implication is clear, as \(z \posnotbelow x\) implies
  \(z \not\below x\). Now suppose that \(x\) is sharp and that we have
  \(y \ll z\). Using interpolation twice, there exist \(u,v \in D\) with
  \(y \ll u \ll v \ll z\). By sharpness of \(x\), we have \(u \ll x\) or
  \(v \not\below x\). If \(u \ll x\), then \(y \ll x\) and we are done; and if
  \(v \not\below x\), then \(z \posnotbelow x\) by
  Lemma~\ref{posnotbelow-criterion}.
\end{proof}

As expected, it suffices to check sharpness condition for basis elements, and in
the algebraic case, an even simpler criterion is available, as we show now.

\begin{lemma}\label{sharp-basis}
  An element \(x\) of a continuous dcpo \(D\) with a basis \(B\) is sharp if and
  only if for every \(a,b \in B\) with \(a \ll b\) we have \(a \ll x\) or
  \(b \not\below x\).
\end{lemma}
\begin{proof}
  This is similar to the proof of Proposition~\ref{sharp-seemingly-stronger}, but we use
  that the interpolants can be found in the basis.
\end{proof}

\begin{proposition}\label{sharp-algebraic}
  An element \(x\) of an algebraic dcpo \(D\) is sharp if and only if for every
  compact \(c \in D\) it is decidable whether \(c \below x\) holds.
\end{proposition}
\begin{proof}
  Suppose that we sharp element \(x\) and let \(c\) be an arbitrary compact
  element. By compactness of \(c\) and sharpness of \(x\), we have \(c \ll x\)
  or \(c \not\below x\), so \(c \below x\) is decidable, as desired.
  Conversely, suppose that \(c \below x\) is decidable for every compact element
  \(c \in D\). We use Lemma~\ref{sharp-basis} applied to the basis of compact
  elements, so let \(a \ll b\) be compact elements. By the decidability
  assumption, either \(a \below x\) or not. In the first case, we get
  \(a \ll x\) as desired by compactness of \(a\). In the second case, we get
  \(b \not\below x\), completing the proof.
\end{proof}

\begin{proposition}\label{em-sharp}
  Assuming excluded middle, every element of any dcpo is sharp.
  The sharp elements of the Sierpi\'nski domain \(\Sierp\) are exactly \(\bot\)
  and \(\top\).
  Hence, if every element of \(\Sierp\) is sharp, then excluded middle
  follows.
\end{proposition}
\begin{proof}
  Let \(x\) be an arbitrary element of a dcpo \(D\) and suppose that \(y\) and
  \(z\) are elements of \(D\) with \(y \ll z\). By excluded middle, we have
  \(z \below x\) or \(z \not\below x\). In the latter case, we are done. And if
  \(z \below x\), then \(y \ll z \below x\), so \(y \ll x\).
  By Proposition~\ref{sharp-algebraic}, an element \(x \in \Sierp\) is sharp if
  and only if \(\top = \set{*} \below x\) is decidable, so the sharp elements of
  \(\Sierp\) are exactly the decidable subsets of \(\set{\ast}\). Hence, the
  only sharp elements of \(\Sierp\) are \(\bot = \emptyset\) and \(\top\).
\end{proof}

Sharpness also allows us to prove a converse to
Proposition~\ref{apart-implies-sbelow}:

\begin{proposition}\label{sbelow-implies-apart}
  If \(x\) is a sharp element of a continuous dcpo \(D\), then \(x \sbelow y\)
  implies \(x \apart y\) for every element \(y \in D\).
\end{proposition}
\begin{proof}
  Suppose that \(x\) is a sharp element of a continuous dcpo \(D\) with
  \(x \sbelow y\). We are going to show that \(y \posnotbelow x\) which implies
  \(x \apart y\) by definition.
  By continuity of \(D\) and our assumption \(x \sbelow y\), it suffices to
  prove that \(d \below \bigsqcup\{x\} \cup \{y \mid y \posnotbelow x\}\) for
  every \(d \ll y\).
  Given such \(d \ll y\), we have \(d \ll x\) or \(y \posnotbelow x\) by
  sharpness of \(x\) and Proposition~\ref{sharp-seemingly-stronger}. In either
  case we get the desired inequality, completing the proof.
\end{proof}

We can relate the notion of sharpness to a notion by
\cite{Spitters2010}~and~\cite[Definition~3.5]{Kawai2017} of a located subset:
A subset \(V\) of a poset \(S\) is \emph{located} if for every \(s,t \in S\)
with \(s \ll t\), we have \(t \in V\) or \(s \not\in V\).

\begin{proposition}\label{sharp-iff-located-nbhds}
  An element \(x\) of a continuous dcpo is sharp if and only if the filter of
  Scott open neighbourhoods of \(x\) is located in the poset of Scott opens of
  \(D\).
\end{proposition}
\begin{proof}
  Suppose first that \(x \in D\) is sharp. We show that the filter
  \(\nbhds (x)\) of Scott open neighbourhoods of \(x\) is located in the poset
  of Scott opens of \(D\). So assume that we have Scott opens \(U\)~and~\(V\)
  with \(U \ll V\). We must show that \(x \in V\) or \(x \not\in U\).
  Using continuity of \(D\), interpolation and Scott openness of \(V\), we can
  prove that the inclusion
  \(V \subseteq \bigcup\set{\upupset e \mid e \in D, \exists_{d \in V}\, d \ll
    e}\) holds.
  Hence, since also \(U \ll V\), there are finitely many elements
  \(d_1,d_2,\dots,d_n \in V\) and \(e_1,e_2,\dots,e_n \in D\) with
  \(d_1 \ll e_1 , \dots , d_n \ll e_n\) such that
  \(U \subseteq \upupset e_1 \cup \dots \cup \upupset e_n\).
  For every \(1 \leq i \leq n\), we have \(d_i \ll x\) or \(e_i \not\below x\)
  by sharpness of \(x\). We can distinguish two cases: either \(d_j \ll x\) for
  some \(1 \leq j \leq n\) or \(e_i \not\below x\) for every
  \(1 \leq i \leq n\). If the former holds, then
  \(x \in \upupset d_j \subseteq V\), so \(x \in V\) and we are done. And if the
  latter is the case, then \(x \not\in U\) as we had
  \(U \subseteq \upupset e_1 \cup \dots \cup \upupset e_n\). Thus, \(x \in V\)
  or \(x \not\in U\), as desired.

  Conversely, suppose that \(\nbhds(x)\) is located and let \(u,v \in D\) with
  \(u \ll v\). We must show that \(u \ll x\) or \(v \not\below x\). By
  interpolation, there exists \(d \in D\) with \(u \ll d \ll v\). Then
  \(\upupset d \ll \upupset u\) in the poset of Scott opens of \(D\). Hence, by
  locatedness of \(\nbhds(x)\), we have \(\upupset u \in \nbhds(x)\) or
  \(\upupset d \not\in \nbhds(x)\). Hence, \(u \ll x\) or \(\lnot (d \ll
  x)\). In the first case, we are done immediately. In the second case,
  assuming \(v \below x\) leads to \(d \ll x\), contradicting
  \(\lnot (d \ll x)\), so \(v \not\below x\) and we are finished too.
\end{proof}

The following theorems give examples of sharp elements.
\begin{therm}\label{basis-is-sharp}
  Let \(D\) be a continuous dcpo with a basis \(B\).
  \begin{enumerate}
  \item Assuming that \(D\) is pointed, the least element of \(D\) is sharp if
    \(B\) satisfies \eqref{equals-bot-decidable}.
  \item Every element of \(B\) is sharp if \(B\) satisfies
    \eqref{way-below-decidable} or \eqref{below-decidable}.
    In particular, in these cases, the sharp elements are a Scott dense subset
    of \(D\) in the sense of Proposition~\ref{basis-is-dense}.
  \end{enumerate}

  If \(D\) is algebraic, then we can reverse the implications for the
  basis of compact elements of \(D\):
  \begin{enumerate}[resume*]
  \item Assuming that \(D\) is pointed, the least element of \(D\) is sharp if
    and only if the set of compact elements of \(D\) satisfies
    \eqref{equals-bot-decidable}.
  \item The compact elements of \(D\) are sharp if and only if the set of
    compact elements of \(D\) satisfies \eqref{way-below-decidable}
    or~\eqref{below-decidable}.
  \end{enumerate}
\end{therm}
\begin{proof}
  (1): Suppose first that \(B\) satisfies \eqref{equals-bot-decidable} and let
  \(a,b\in B\) with \(a \ll b\) be arbitrary. By assumption, either \(a = \bot\)
  or \(a \neq \bot\). If \(a = \bot\), then \(a \ll \bot\), by compactness of
  \(\bot\) and we are done. If \(a \neq \bot\), then \(a \not\below \bot\), so
  \(b \not\below \bot\) and we are finished too.
  (2): Suppose that \(x \in B\) and let \(a,b \in B\) with \(a \ll b\).  If
  \(B\) satisfies \eqref{way-below-decidable}, then either \(a \ll x\) or
  \(\lnot (a\ll x)\). In the first case we are done. In the second case we get
  \(b \not\below x\), for if \(b \below x\), then \(a \ll b \below x\), so
  \(a \ll x\), contradicting our assumption.
  If \(B\) satisfies \eqref{below-decidable}, then either \(b \below x\) or
  \(b \not\below x\). In the second case we are done. In the first case we get
  \(a \ll b \below x\), so \(a \ll x\) and we are done too.
  Scott density of the sharp elements now follows at once from
  Proposition~\ref{basis-is-dense}.

  (3): Suppose that \(\bot\) is sharp and let \(c \in D\) be compact. By
  Proposition~\ref{sharp-algebraic}, it is decidable whether \(c \below \bot\)
  holds, so the set of compact elements of \(D\) satisfies
  \eqref{equals-bot-decidable}.
  (4): If every compact element is sharp, then the set of compact elements
  satisfies \eqref{below-decidable} by Proposition~\ref{sharp-algebraic}.
\end{proof}

\begin{therm}\label{sharpness-product-exponential}\hfill
  \begin{enumerate}
  \item An element \((x,y)\) of the product of two dcpos \(D\) and \(E\) is
    sharp if and only if \(x\) and \(y\) are sharp elements of \(D\) and \(E\)
    respectively.
  \item An element \(f \colon D \to E\) of the exponential of two bounded
    complete, algebraic dcpos \(D\) and~\(E\) is sharp if and only if \(f(c)\)
    is sharp in \(E\) for every compact element \(c : D\).
  \end{enumerate}
\end{therm}
\begin{proof}
  (1): Straightforward as the order is component-wise.
  (2): Suppose that \(f \colon D \to E\) is sharp and let \(c \in D\) be an
  arbitrary compact element. By Proposition~\ref{sharp-algebraic} we have to
  show that \(e \below f(c)\) is decidable for every compact element
  \(e \in E\).
  But \(e \below f(c)\) holds exactly when \(\steppa{c \To e} \below f\) by
  Lemma~\ref{step-function-below}, which is decidable by sharpness of \(f\) and
  Proposition~\ref{sharp-algebraic} again.
  For the converse, it suffices, by
  Propositions~\ref{dec-closed-under-exponentials}~and~\ref{sharp-algebraic}, to
  show that \(\steppa{a \To b} \below f\) is decidable for \(a \in D\) and
  \(b \in E\) compact elements.
  But again this reduces to decidability of \(b \below f(a)\) which is implied
  by the assumed sharpness of \(f(a)\).
\end{proof}

Finally, we see that sharp elements are well-behaved with respect to the
intrinsic apartness, because we get tightness and cotransitivity when
restricting to sharp elements.

\begin{therm}\label{sharp-tight-cotransitive}\hfill
  \begin{enumerate}
  \item If \(y\) is a sharp element of a continuous \(D\), then
    \(\lnot (x \posnotbelow y)\) implies \(x \below y\) for every \(x \in D\).
    In particular, the intrinsic apartness on a continuous dcpo \(D\) is tight
    on sharp elements.
  \item The intrinsic apartness on a continuous dcpo \(D\) is cotransitive with
    respect to sharp elements in the following sense: for every \(x,y \in D\)
    and sharp element \(z \in D\), we have
    \(x \apart y \to \pa*{x \apart z \vee y \apart z}\).
  \end{enumerate}
\end{therm}
\begin{proof}
  (1): Suppose that \(y\) is sharp with \(\lnot (x \posnotbelow y)\). We use
  Lemma~\ref{below-in-terms-of-way-below} to prove \(x \below y\). So let
  \(u \in D\) with \(u \ll x\). Using interpolation, there exists \(v \in D\)
  with \(u \ll v \ll x\). So by sharpness of \(y\), we have \(u \ll y\) or
  \(v \not\below y\). If \(u \ll y\), then we are done; and if
  \(v \not\below y\), then \(x \posnotbelow y\), contradicting our assumption.
  Now if \(x\) and \(y\) are both sharp and \(\lnot (x \apart y)\), then
  \(\lnot (x \posnotbelow y)\) and \(\lnot (y \posnotbelow x)\). Hence,
  \(x \below y\) and \(y \below x\) by the above, from which \(x = y\) follows
  using antisymmetry.

  (2): Let \(x,y \in D\) be such that \(x \apart y\) and \(z \in D\) any
  sharp element. Assume without loss of generality that
  \(x \posnotbelow y\). By Lemma~\ref{posnotbelow-criterion} and interpolation,
  there exist \(u,v \in D\) with \(u \ll v \ll x\) and \(u \not\below y\). Since
  \(z\) is sharp, we have \(u \ll z\) or \(v \not\below z\). If
  \(u \ll z\), then by Lemma~\ref{posnotbelow-criterion}, we have
  \(z \posnotbelow y\), so \(z \apart y\) and we are done.  If
  \(v \not\below z\), then \(x \posnotbelow z\) by
  Lemma~\ref{posnotbelow-criterion}, so \(x \apart z\), finishing the proof.
\end{proof}

By Theorem~\ref{sharp-basis}, the basis elements of a continuous dcpo are sharp
in many examples of interest. The following section provides a very different
source of sharp elements.

\section{Strongly maximal elements}\label{sec:strongly-maximal-elements}
\cite{Smyth2006} explored, in a classical setting, the notion of a
\emph{constructively maximal} element, adapted from~\citep{MartinLof1970}. In an
unpublished manuscript, \cite{Heckmann1998} arrived at an equivalent notion,
assuming excluded middle as noted in \cite[Section~8]{Smyth2006}, and called it
\emph{strong maximality}. Whereas Smyth works directly with abstract bases and
rounded ideal completions, we instead work with continuous dcpos. We use a
simplification of Smyth's definition, but follow Heckmann's terminology. We show
that every strongly maximal element is sharp, compare strong maximality to
maximality highlighting connections to sharpness, and study the subspace of
strongly maximal elements.

\begin{definition}[Hausdorff separation]
  Two points \(x\) and \(y\) of a dcpo \(D\) are \emph{Hausdorff separated} if
  we have disjoint Scott open neighbourhoods of \(x\) and \(y\)
  respectively.
\end{definition}

\begin{definition}[Strong maximality]\label{def:strongly-maximal}
  An element \(x\) of a continuous dcpo \(D\) is \emph{strongly maximal} if for
  every \(u,v \in D\) with \(u \ll v\), we have \(u \ll x\) or \(v\) and \(x\)
  are Hausdorff separated.
\end{definition}

The following gives another source of examples of sharp elements.
\begin{proposition}\label{sharp-if-strongly-maximal}
  Every strongly maximal element of a continuous dcpo is sharp.
\end{proposition}
\begin{proof}
  Let \(x\) be a strongly maximal element of a continuous dcpo \(D\) and let
  \(u,v\in D\) be such that \(u \ll v\). By strong maximality, we have
  \(u \ll x\) or \(v\) and \(x\) are Hausdorff separated. In the first case we
  are done immediately. And if \(v\) and \(x\) are Hausdorff separated, then
  \(v \not\below x\), because Scott opens are upper sets. Hence, \(x\) is sharp,
  as desired.
\end{proof}

\begin{corollary}\hfill
  \begin{enumerate}
  \item The intrinsic apartness on a continuous dcpo \(D\) is tight on strongly
    maximal elements.
  \item The intrinsic apartness on a continuous dcpo \(D\) is cotransitive with
    respect to strongly maximal elements in the following sense: for every
    \(x,y \in D\) and strongly maximal element \(z \in D\), we have
    \({x \apart y \to \pa*{x \apart z \vee y \apart z}}\).
  \end{enumerate}
\end{corollary}
\begin{proof}
  By Proposition~\ref{sharp-if-strongly-maximal} and
  Theorem~\ref{sharp-tight-cotransitive}.
\end{proof}

Similarly to sharpness, the condition of Definition~\ref{def:strongly-maximal}
can be restricted to the basis and simplifies in the case of algebraic dcpos:

\begin{lemma}\label{strongly-maximal-basis}
  If a continuous dcpo \(D\) has a basis \(B\), then an element \(x \in D\) is
  strongly maximal if~and~only~if for every \(a,b \in B\) with \(a \ll b\),
  we have \(a \ll x\) or \(b\) and \(x\) are Hausdorff separated.
\end{lemma}
\begin{proof}
  The left-to-right implication is immediate. Conversely, suppose that \(x\)
  satisfies the criterion in the lemma and let \(u,v \in D\) with \(u \ll
  v\). Using interpolation twice, there are \(a,b \in B\) such that
  \(u \ll a \ll b \ll v\). By assumption, we have \(a \ll x\) or \(b\) and \(x\)
  are Hausdorff separated. If \(a \ll x\), then \(u \ll a \ll x\) and we are
  done. And if \(b\) and \(x\) are Hausdorff separated, then so are \(v\) and
  \(x\), because Scott opens are upper sets.
\end{proof}
\begin{lemma}\label{strongly-maximal-algebraic}
  An element \(x\) of an algebraic dcpo \(D\) is strongly maximal if and only if
  for every compact element \(c \in D\), either \(c \below x\) or \(c\) and \(x\) are
  Hausdorff separated.
\end{lemma}
\begin{proof}
  The left-to-right implication is clear. So suppose that \(x\) satisfies the
  condition in the lemma. We~use Lemma~\ref{strongly-maximal-basis} on the basis
  \(B\) of compact elements. So suppose that we have \(a,b\in B\) with
  \(a \ll b\). By assumption and compactness of \(a\), either \(a \ll x\) or
  \(a\) and \(x\) are Hausdorff separated. If \(a \ll x\), then we are done. And
  if \(a\) and \(x\) are Hausdorff separated, then so are \(b\) and \(x\),
  because Scott opens are upper sets.
\end{proof}

Smyth's formulation of strong maximality~\cite[Definition~4.1]{Smyth2006}
(called \emph{constructive maximality} there) is somewhat more involved than
ours, but it is equivalent, as Proposition~\ref{Smyth-strong-equiv} shows.

\begin{definition}[Refinement \(x \refined y\)]\label{def:refine}
  We say that two elements \(x\) and \(y\) of a dcpo \(D\) can be
  \emph{refined}, written \(x \refined y\), if there exists \(z \in D\) with
  \(x \ll z\) and \(y \ll z\).
\end{definition}
On~\cite[p.~362]{Smyth2006}, refinement is denoted by
\(x \mathrel{\uparrow} y\), but we prefer \(x \refined y\), because one might
want to reserve \(x \mathrel{\uparrow} y\) for the weaker statement that there
exists \(z \in D\) with \(x,y \below z\).

\begin{lemma}\label{Hausdorff-separated-criterion}
  Two elements \(x\) and \(y\) of a continuous dcpo \(D\) are Hausdorff
  separated if and only if there exist
  \(a,b \in D\) with \(a \ll x\) and \(b \ll y\) such that
  \(\lnot \pa*{a \refined b}\).
  Moreover, if \(D\) has a basis \(B\), then \(x\) and \(y\) are Hausdorff
  separated if and only if there exist such elements \(a\) and \(b\) in \(B\).
\end{lemma}
\begin{proof}
  Suppose that \(D\) has a basis \(B\) (which may be all of \(D\)) and let \(x\)
  and \(y\) be Hausdorff separated elements of \(D\). Then there are two
  disjoint basic opens containing \(x\) and \(y\) respectively. Hence, by
  Lemma~\ref{basis-basis}, there exist \(a,b \in B\) such that
  \(x \in \upupset a\), \(y \in \upupset b\) and
  \(\upupset a \cap \upupset b = \emptyset\). The latter says exactly that
  \(\lnot \pa*{a \refined b}\).
  Conversely, if we have \(a,b\in B\) with \(a \ll x\), \(b \ll y\) and
  \(\lnot \pa*{a \refined b}\), then \(\upupset a\) and \(\upupset b\) are
  disjoint Scott opens containing \(x\)~and~\(y\) respectively.
\end{proof}
In \cite[p.~362]{Smyth2006}, the condition in
Lemma~\ref{Hausdorff-separated-criterion} is taken as a definition (for basis
elements) and such elements \(a\)~and~\(b\) are said to lie apart and this
notion is denoted by \(a \mathrel{\sharp} b\).
We now translate Definition 4.1 of \citep{Smyth2006} from ideal completions
of abstract bases to continuous dcpos.
\begin{definition}[Smyth maximality]
  An element \(x\) of a continuous dcpo \(D\) is \emph{Smyth maximal} if for
  every \(u,v \in D\) with \(u \ll v\), there exists \(d \ll x\) such that
  \(u \ll d\) or the condition in Lemma~\ref{Hausdorff-separated-criterion}
  holds for \(v\) and \(d\).
\end{definition}

\begin{proposition}\label{Smyth-strong-equiv}
  An element \(x\) of a continuous \(D\) is strongly maximal if and only if
  \(x\) is Smyth maximal.
\end{proposition}
\begin{proof}
  The right-to-left implication is straightforward thanks to
  Lemma~\ref{Hausdorff-separated-criterion}. So suppose that \(x\) is strongly
  maximal and that we have \(u \ll v\). Then \(u \ll x\) holds or \(v\) and
  \(x\) are Hausdorff separated.
  Suppose first that \(u \ll x\). Then by interpolation there exists \(d \in D\)
  with \(u \ll d \ll x\) and we are done.
  Now suppose that \(v\) and \(x\) are Hausdorff separated. Then we have
  disjoint Scott opens \(V\) and \(U\) containing \(v\) and \(x\) respectively.
  By continuity of \(D\) and Scott openness of \(U\), there exists \(d \ll x\)
  with \(d \in U\). But the opens \(V\) and \(U\) then witness Hausdorff
  separation of \(v\) and \(d\) too, so that
  Lemma~\ref{Hausdorff-separated-criterion} finishes the proof.
\end{proof}

Finally, we how strong maximality is preserved by taking products and
exponentials.

\begin{proposition}\label{strong-maximality-product-exponential}\hfill
  \begin{enumerate}
  \item An element \((x,y)\) of the product of two dcpos \(D\) and \(E\) is
    strongly maximal if and only if both \(x\) and \(y\) are strongly maximal
    elements of \(D\) and \(E\) respectively.
  \item An element \(f \colon D \to E\) of the exponential of two bounded
    complete, algebraic dcpos \(D\) and~\(E\) is strongly maximal if \(f(c)\) is
    strongly maximal in \(E\) for every compact element \(c : D\).
  \end{enumerate}
\end{proposition}
\begin{proof}
  (1) is easily verified, so we only give the details for (2). Suppose that
  \(f(c)\) is strongly maximal for every compact element \(c \in D\). We show
  that \(f\) is strongly maximal in the exponential~\(E^D\). By
  Lemma~\ref{strongly-maximal-algebraic} and
  Proposition~\ref{dec-closed-under-exponentials}, it is enough to prove that
  for every two compact elements \(a \in D\) and \(b \in E\), we have
  \(\steppa{a \To b} \below f\) or \(\steppa{a \To b}\) and \(f\) can be
  Hausdorff separated.
  Given such elements \(a\) and \(b\), we have either \(b \below f(a)\) or \(b\)
  and \(f(a)\) can be Hausdorff separated by strong maximality of \(f(a)\).
  In the first case we get \(\steppa{a \To b} \below f\) by
  Lemma~\ref{step-function-below}, as desired.
  In the second case, we get, using Lemma~\ref{basis-basis}, compact elements
  \(c,d \in E\) such that \(f(a) \in \upupset c\) and \(b \in \upupset d\) while
  \(\upupset c \cap \upupset d = \emptyset\).
  But now \(\upupset \steppa{a \To c}\) and \(\upupset \steppa{a \To d}\) are
  disjoint open neighbourhoods of respectively \(f\) and \(\steppa{a \To b}\) in
  \(E^D\), finishing the proof.
\end{proof}

Note that, unlike in Theorem~\ref{sharpness-product-exponential}, the
converse to Item (2) of Proposition~\ref{strong-maximality-product-exponential}
does not hold: the identity on \(\lifting(\Nat)\) is strongly maximal, but
\(\bot\) is not. Here \(\lifting(\Nat)\) denotes the free pointed dcpo on the
set of natural numbers (see Section~\ref{lifting-Cantor}), which, classically,
is simply the flat dcpo \(\Nat \cup \{\bot\}\).

\subsection{Maximality and strong maximality}\label{subsec:max-strong-max}
The name strong maximality is justified by the following observation:
\begin{proposition}[cf.\ Proposition~4.2 in \citep{Smyth2006}]\label{maximal-if-strongly-maximal}
  Every strongly maximal element of a continuous dcpo \(D\) is maximal, i.e.\ if
  \(x \in D\) is strongly maximal, then \(x \below y\) implies \(x = y\) for
  every element \(y \in D\).
\end{proposition}
\begin{proof}
  Let \(x\) be a strongly maximal element of a continuous \(D\) and suppose that
  we have \(y \in D\) with \(x \below y\). We need to prove that \(y \below x\).
  We do so using Lemma~\ref{below-in-terms-of-way-below}. So let \(u \in D\) be
  such that \(u \ll y\). Our goal is to prove that \(u \ll x\). By strong
  maximality of \(x\), we have \(u \ll x\) or \(x\) and \(y\) are Hausdorff
  separated. In the first case we are done, while the second case is impossible,
  because \(x \below y\) and Scott opens are upper sets.
\end{proof}

In the presence of excluded middle, \cite[Corollary~4.4]{Smyth2006} tells us
that the converse of Proposition~\ref{maximal-if-strongly-maximal} is true if
and only if the {Lawson~condition}~\citep{Smyth2006,Lawson1997} holds for the
dcpo (the Scott and Lawson topologies coincide on the subset of maximal
elements).
Without excluded middle, the situation is subtler and involves sharpness, as we
show in Theorem~\ref{Lawson-strongly-maximal}.

The following lemma is a technical device for constructing maximal elements
that, constructively, fail to be strongly maximal. Besides
Proposition~\ref{maximal-strongly-maximal-wem-alt}, it also finds application in
Propositions~\ref{seq-strongly-max-wem} and \ref{Dedekind-max-strong-max-wem}.

\begin{lemma}\label{maximal-strongly-maximal-wem}
  Suppose that we have elements \(x\) and \(y\) of a continuous dcpo \(D\) such
  that
  \begin{enumerate}
  \item\label{both-strongly-max} \(x\) and \(y\) are both strongly maximal,
  \item\label{glb} \(x\) and \(y\) have a greatest lower bound \(x \glb y\) in \(D\), and
  \item\label{apart} \(x\) and \(y\) are intrinsically apart.
  \end{enumerate}
  Then, for any proposition \(P\), the supremum \(\dirsup S\) of the directed
  subset
  \[
    S \coloneqq \set{x \glb y} \cup \set{x \mid \lnot P} \cup \set{y \mid
      \lnot\lnot P}
  \] is maximal, but \(\dirsup S\) is sharp if and only if \(\lnot P\) is
  decidable.  Hence, \(\dirsup S\) is maximal, but if \(\dirsup S\) is strongly
  maximal, then \(\lnot P\) is decidable.
\end{lemma}
\begin{proof}
  We start by showing that \(\dirsup S\) is maximal. Suppose that we have
  \(z \in D\) with \(\dirsup S \below z\). We use
  Lemma~\ref{below-in-terms-of-way-below} to show that \(z \below \dirsup
  S\). So let \(u \in D\) be such that \(u \ll z\). Our goal is to show that
  \(u \below \dirsup S\).
  By strong maximality of \(x\), we have \(u \ll x\) or \(z\) and \(x\) are
  Hausdorff separated. As \(y\) is also strongly maximal, we have \(u \ll y\) or
  \(z\) and \(y\) are Hausdorff separated. We thus distinguish four cases:
  \begin{itemize}
  \item (\(u \ll x\) and \(u \ll y\)): In this case, \(u \below {x \glb y}\) by
    assumption~(2). Hence, \(u \below {x \glb y} \below \dirsup S\), and we are
    done in this case.
  \item (\(u \ll x\) and \(z\) and \(y\) are Hausdorff separated): We claim that
    \(\lnot P\) must hold in this case. For suppose that \(P\) were true. Then
    \(y = \dirsup S \below z\), contradicting that \(z\) and \(y\) are Hausdorff
    separated. Thus, \(\lnot P\) holds. But then \(\dirsup S = x\) and
    \(u \ll x = \dirsup S\), so \(u \below \dirsup S\), as wished.
  \item (\(z\) and \(x\) are Hausdorff separated and \(u \ll y\)): This is
    similar to the above case, but now \(\lnot\lnot P\) must hold, so that
    \({u \ll y = \dirsup S}\), so \(u \below \dirsup S\), as desired.
  \item (\(z\) and \(x\) are Hausdorff separated as are \(z\) and \(y\)): This
    case is impossible, as it implies, using an argument similar to that used in
    the previous two cases, that both \(\lnot P\) and \(\lnot\lnot P\) must
    hold.
  \end{itemize}
  So in all cases, we get the desired \(u \below \dirsup S\) and \(\dirsup S\)
  is seen to be maximal.

  Now suppose that \(\dirsup S\) is sharp. We show that we can decide \(\lnot P\).
  By assumption, \(x\) and \(y\) are apart. We assume that \(x \posnotbelow y\);
  the case for \(y \posnotbelow x\) is similar. By
  Lemma~\ref{posnotbelow-criterion}, there exists \(d \in D\) with \(d \ll x\)
  and \(d \not\below y\). By assumed sharpness of \(\dirsup S\), we have
  \(d \ll \dirsup S\) or \(x \not\below \dirsup S\). If \(d \ll \dirsup S\),
  then we claim that \(\lnot P\) holds. For suppose for a contradiction that
  \(P\) holds, then \(d \ll \dirsup S = y\), contradicting \(d \not\below
  y\). Now suppose that \(x \not\below \dirsup S\).
  Since \(\lnot P\) implies \(x = \dirsup S\), we see that \(\lnot\lnot P\) must
  hold in this case. Hence, \(\lnot P\) is decidable.

  Conversely, suppose that \(\lnot P\) is decidable. If \(\lnot P\) holds, then
  \(\dirsup S = x\), so \(\dirsup S\) must be sharp since \(x\) is by
  Proposition~\ref{sharp-if-strongly-maximal}. And if \(\lnot\lnot P\) holds,
  then \(\dirsup S = y\), and sharpness of \(\dirsup S\) follows from
  Proposition~\ref{sharp-if-strongly-maximal} and strong maximality of \(y\).
  The final claim follows from Proposition~\ref{sharp-if-strongly-maximal} and
  the fact that we just proved that sharpness of \(\dirsup S\) implies
  decidability of \(\lnot P\).
\end{proof}

\begin{proposition}\label{maximal-strongly-maximal-wem-alt}
  Let \(P\) be the poset with exactly three elements \(\bot \leq 0,1\) and \(0\)
  and \(1\) unrelated.  If every maximal element of the algebraic dcpo
  \(\Idl(P)\) is strongly maximal, then weak excluded middle follows.
  In the other direction, if excluded middle holds, then every maximal element
  of \(\Idl(P)\) is strongly maximal.
\end{proposition}
\begin{proof}
  We apply Lemma~\ref{maximal-strongly-maximal-wem} with
  \(x \coloneqq \set{\bot,0}\) and \(y \coloneqq \set{\bot,1}\) in
  \(\Idl(P)\). Observe that \(x\) and \(y\) are apart and that they have a
  greatest lower bound, namely \(\set{\bot}\). Thus it remains to show that they
  are both strongly maximal. We do so for \(x\); the proof for \(y\) is
  similar. We apply Lemma~\ref{strongly-maximal-algebraic} with the basis of
  compact elements \(\set{\dset p \mid p \in P}\) and so distinguish three
  cases:
  \begin{itemize}
  \item (\(p = \bot\)): Note that \(\dset \bot = \set{\bot} \ll x\).
  \item (\(p = 0\)): Note that \(\dset 0 = \set{\bot,0} = x \ll x\).
  \item (\(p = 1\)): The elements \(\dset 1 = \set{\bot,1} = y\) and \(x\) are
    Hausdorff separated, since \(\set{x} = \upupset x\) and
    \(\set{y} = \upupset y\) are disjoint Scott opens containing \(x\) and \(y\)
    respectively.
  \end{itemize}
  Thus, \(x\) and \(y\) are both strongly maximal, so weak excluded middle
  follows by Lemma~\ref{maximal-strongly-maximal-wem}.
  If excluded middle holds, then we can easily prove that all maximal elements
  of \(\Idl(P)\) are strongly maximal.
\end{proof}

Theorem~4.3 of \citep{Smyth2006} states that an element \(x\) of a continuous
dcpo is strongly maximal if and only if every Lawson neighbourhood of \(x\)
contains a Scott neighbourhood of \(x\).
This requirement on neighbourhoods is, assuming excluded middle, equivalent to
the Lawson condition (the Scott topology and the Lawson topology coincide on the
subset of maximal elements), as used in~\cite[Corollary~4.4]{Smyth2006} and
proved in \cite[Lemma~V\nobreakdash-6.5]{GierzEtAl2003}. Inspecting the proof
of~\cite{GierzEtAl2003}, we believe that excluded middle is essential. However,
we can still prove a constructive analogue of \cite[Theorem~4.3]{Smyth2006}, but
it requires a positive formulation of the subbasic opens of the Lawson topology,
using the relation \({\posnotbelow}\) rather than \({\not\below}\).

\begin{definition}[Lawson topology]\label{Lawson-subbasics}
  The subbasic \emph{Lawson closed} subsets of a dcpo \(D\) are the Scott closed
  subsets and the upper sets of the form \(\upset x\) for \(x \in D\).
  The subbasic \emph{Lawson opens} are the Scott opens and the sets of
  the form \(\set{y \in D \mid x \posnotbelow y}\) for \(x \in D\).
\end{definition}

In the presence of excluded middle, the subset
\(\set{y \in D \mid x \posnotbelow y}\) is equal to \(D \setminus \upset x\), so
with excluded middle the above definition is equivalent to the classical
definition of the subbasic Lawson opens as
in~{\cite[pp.~211--212]{GierzEtAl2003}}.

\begin{therm}[cf.~{\cite[Theorem~4.3]{Smyth2006}}]
  \label{Lawson-strongly-maximal}
  An element \(x\) of a continuous dcpo is strongly maximal if and only if \(x\)
  is sharp and every Lawson neighbourhood of \(x\) contains a Scott
  neighbourhood of \(x\).
\end{therm}
\begin{proof}
  Let \(x\) be a strongly maximal element of a continuous dcpo \(D\).
  Then \(x\) is sharp by Proposition~\ref{sharp-if-strongly-maximal}.
  It remains to show that every Lawson neighbourhood of \(x\) contains a Scott
  neighbourhood of \(x\). It suffices to do so for the subbasic Lawson open
  neighbourhoods of~\(x\) and in particular for those of the form
  \(\set{d \in D \mid y \posnotbelow d}\) for \(y \in D\). So suppose that we
  have \(y \in D\) with \(y \posnotbelow x\). Then there exists \(u \ll y\) with
  \(u \not\below x\).
  By strong maximality of \(x\), we have \(u \ll x\) or \(y\) and \(x\) are
  Hausdorff separated. Since \(u \not\below x\), the first case \(u \ll x\) is
  impossible. Hence, we have disjoint Scott opens \(U\)~and~\(V\) containing
  \(x\) and \(y\) respectively.
  We claim that \(U \subseteq \set{d \in D \mid y \posnotbelow d}\), which would
  finish the proof. Since \(V\) is a Scott open neighbourhood of \(y\), we can
  use continuity of \(D\) to find \(v \ll y\) with \(v \in V\). Now if
  \(u \in U\), then because \(U\) and \(V\) are disjoint and Scott opens are
  upper sets, we must have \(v \not\below u\). Hence, \(v\) witnesses that
  \(y \posnotbelow u\), so that
  \(U \subseteq \set{d \in D \mid y \posnotbelow d}\).

  Conversely, suppose that \(x\) is a sharp element of a continuous dcpo \(D\)
  such that every Lawson neighbourhood of \(x\) contains a Scott neighbourhood
  of \(x\). We prove that \(x\) is strongly maximal. So let \(u,v\in D\) such
  that \(u \ll v\). By interpolation, there exist \(u',v' \in D\) with
  \(u \ll u' \ll v' \ll v\). Because \(x\) is sharp, we have \(u \ll x\) or
  \(u' \not\below x\). In the first case we are done, so assume that
  \(u' \not\below x\). Then \(u'\) witnesses \(v' \posnotbelow x\).
  Hence, \(\set{d \in D \mid v' \posnotbelow d}\) is a Lawson neighbourhood of
  \(x\). So by assumption on \(x\), there exists a Scott open neighbourhood
  \(U\) of \(x\) contained in \(\set{d \in D \mid v' \posnotbelow d}\).
  But then \(U\)~and~\(\upupset v'\) are disjoint Scott opens containing \(x\)
  and \(v\) respectively. So \(x\) and \(v\) are Hausdorff separated.
\end{proof}

By Proposition~\ref{em-sharp}, every element is sharp if excluded middle is
assumed, so in that case, we can drop the requirement that \(x\) is
sharp. Hence, in the presence of excluded middle, we recover
\cite[Theorem~4.3]{Smyth2006} from Theorem~\ref{Lawson-strongly-maximal}.

\subsection{The subspace of strongly maximal elements}
The classical interest in strong maximality comes from the fact that, while the
subspace of maximal elements may fail to be
Hausdorff~\cite[Example~4]{Heckmann1998}, the subspace of strongly maximal
elements with the relative Scott topology is both Hausdorff and
regular~\cite[Theorem~4.6]{Smyth2006}. We offer constructive proofs of these
claims, with the proviso that the Hausdorff condition is formulated with respect
to the intrinsic apartness.
\begin{proposition}
  The subspace of strongly maximal elements of a continuous dcpo \(D\) with the
  relative Scott topology is Hausdorff, i.e.\ if \(x\)~and~\(y\) are strongly
  maximal, then \(x \apart y\) if and only if there are disjoint Scott opens
  \(U\) and \(V\) such that \(x \in U\) and \(y \in V\).
\end{proposition}
\begin{proof}
  Let \(x\) and \(y\) be strongly maximal elements. The right-to-left
  implication is immediate by definition of the intrinsic apartness. Now suppose
  that \(x \apart y\). Using Lemma~\ref{posnotbelow-criterion}, we can assume
  without loss of generality that we have \(d \in D\) with \(d \ll x\) and
  \(d \not\below y\). Since \(y\) is strongly maximal, we have \(d \ll y\) or
  \(x\) and \(y\) are Hausdorff separated. But \(d \ll y\) is impossible,
  because \(d \not\below y\). So \(x\)~and~\(y\) are Hausdorff separated, as
  desired.
\end{proof}

\begin{proposition}
  The subspace of strongly maximal elements of a continuous dcpo \(D\) with the
  relative Scott topology is regular, i.e.\ every Scott neighbourhood of a point
  \(x \in D\) contains a Scott closed neighbourhood of \(x\).
\end{proposition}
\begin{proof}
  We adapt the proof of \cite[Theorem~4.6~(i)]{Smyth2006}, avoiding the use of
  the axiom of choice.
  Write \(\SMax(D)\) for the subspace of strongly maximal elements of \(D\) and
  let \(x\) be a strongly maximal element. It suffices to prove the claim for
  basic Scott open neighbourhoods of \(x\). So suppose that \(x \in \upupset u\)
  for some \(u \in D\). By interpolation, there exists \({v \in D}\) such that
  \({x \in \upupset v \subseteq \upupset u}\). We will construct a Scott closed
  subset \(C\) such that \(\upupset v \subseteq C\) and
  \[
    C \cap \SMax(D) \subseteq \upupset u \cap \SMax(D).
  \]
  Define the Scott open subset \(V\) by
  \[
    V \coloneqq \bigcup\set{U \subseteq D \mid U \text{ is Scott open and }
      \upupset v \cap U = \emptyset}
  \]
  and the Scott closed subset \(C \coloneqq \lnot V = D \setminus V\).

  Note that if \(y \in V\), then \(y \in U\) for some Scott open \(U\) with
  \(\upupset v \cap U = \emptyset\), so that \(y \not\in \upupset v\). Hence,
  \(V \subseteq \lnot \upupset v\), so
  \(\upupset v \subseteq \lnot\lnot \upupset v \subseteq \lnot V = C\). Thus,
  \(x \in \upupset v \subseteq C\), so \(C\) is a Scott closed neighbourhood of
  \(x\), as desired.

  It remains to show that
  \(C \cap \SMax(D) \subseteq \upupset u \cap \SMax(D)\). So let \(y \in C\) be
  strongly maximal. Then \(u \ll y\) holds or \(v\) and \(y\) are Hausdorff
  separated. If \(u \ll y\), then \(y \in \upupset u\) and we are done.  So
  suppose that \(v\) and \(y\) are Hausdorff separated. We prove that this is
  impossible. Assume for a contradiction that we have disjoint Scott opens
  \(U_v\)~and~\(U_y\) containing \(v\)~and~\(y\) respectively. Since \(U_v\) is
  an upper set, we have \(\upupset v \subseteq U_v\). As \(U_v\) and \(U_y\) are
  disjoint, we thus get \(\upupset v \cap U_y = \emptyset\). Hence,
  \(y \in U_y \subseteq V\), contradicting \(y \in C\).
\end{proof}

\section{Examples}\label{sec:examples}
In this final section before the conclusion, we illustrate the foregoing
notions of intrinsic apartness, sharpness and strong maximality, by studying
several natural examples. The first three examples are generalized domain
environments in the sense of~\cite{Heckmann1998}: we consider dcpos
\(D\) and topological spaces \(X\) such that \(X\) \emph{embeds} into the
subspace of maximal elements of \(D\). In fact, we will see that \(X\) is
homeomorphic to the subspace of \emph{strongly} maximal elements of \(D\).
Specifically, we will consider Cantor space, Baire space and the real line.
The penultimate example shows that sharpness characterizes exactly those lower
reals that are located.
The final example showcases another natural embedding of Cantor space into a
dcpo constructed using exponentials and the lifting monad.

\subsection{The Cantor and Baire domains}\label{sec:Cantor-and-Baire-domains}
Fix an inhabited set \(A\) with decidable equality. Typically, we will be
interested in \(A = \Two \coloneqq \set{0,1}\) and \(A = \Nat\).
Recall that a \emph{finite sequence} on \(A\) is a function
\(\set{0,1,\dots,{n-1}} \to A\) for a natural number \(n \in \Nat\) (for
\(n = 0\) we get the empty sequence).
An \emph{infinite sequence} on \(A\) is simply a function \(\Nat \to A\).

\begin{definition}[\(\Seqspace\) and \(\Seqdom\)]\hfill
  \begin{enumerate}
  \item We write \(\pa*{\Finseq , \preceq}\) for the poset of finite sequences
    on \(A\) ordered by prefix and \(\Seqdom\) for the ideal completion of
    \(\pa*{\Finseq , \preceq}\), which is an algebraic dcpo by
    Lemma~\ref{reflexive-algebraic}.
  \item For an infinite sequence \(\alpha\), we write \(\initseg{\alpha}{n}\)
    for the first \(n\) elements of \(\alpha\).  Given a finite
    sequence~\(\sigma\), we write \(\sigma \prec \alpha\) if \(\sigma\) is an
    initial segment of \(\alpha\).
  \item We write \(\Seqspace\) for the space \(A^\Nat\) of infinite sequences on
    \(A\) with the product topology, taking the discrete topologies on \(\Nat\)
    and \(A\). The sets \(\set{\alpha \in \Seqspace \mid \sigma \prec \alpha}\)
    with \(\sigma\) a finite sequence form a basis of opens for
    \(\Seqspace\). Furthermore, the space \(\Seqspace\) has a natural notion of
    apartness:
    \(\alpha \apart_{\Seqspace} \beta \iff \exists_{n \in
        \Nat}\,\initseg{\alpha}{n} \neq \initseg{\beta}{n}\).
  \end{enumerate}
\end{definition}

\begin{definition}[Cantor and Baire space/domain]
  If \(A = \Two\), then \(\Seqspace\) is Cantor space and we call
  \(\Seqdom\) the \emph{Cantor domain}; and if we take \(A = \Nat\), then
  \(\Seqspace\) is Baire space and we call \(\Seqdom\) the \emph{Baire domain}.
\end{definition}

\begin{definition}[\(\iota\)]
  We define an injection \(\iota \colon \Seqspace \hookrightarrow \Seqdom\) by
  \(\iota(\alpha) \coloneqq \bigcup_{\sigma\prec\alpha} \dset\sigma = \set{\tau
    \in \Finseq \mid \tau \prec \alpha}\).
\end{definition}

\begin{therm}\label{seq-iota-strongly-maximal}
  The image of \(\iota\) is exactly the subset of strongly maximal elements of
  \(\Seqdom\).
\end{therm}
\begin{proof}
  We first prove that \(\iota(\alpha)\) is strongly maximal for every infinite
  sequence \(\alpha\), using
  Lemmas~\ref{strongly-maximal-algebraic} and \ref{reflexive-algebraic}.  So let
  \(\sigma\) be a finite sequence. Since \(A\) has decidable equality and
  \(\sigma\) is finite, either \(\sigma \prec \alpha\) or not. If
  \(\sigma \prec \alpha\), then
  \(\dset\sigma = \set{\tau \in \Finseq \mid \tau \preceq \sigma} \ll
  \iota(\alpha)\) and we are done. So suppose that \(\sigma \nprec
  \alpha\). Writing \(n\) for the length of \(\sigma\), we see that
  \(\sigma \neq \initseg{\alpha}{n}\). Therefore, the Scott opens
  \(\upupset\pa*{\dset \sigma}\) and \(\upupset\pa*{\dset \initseg{\alpha}{n}}\)
  must be disjoint, while the former contains \(\sigma\) and the latter contains
  \(\alpha\), completing the proof that \(\iota(\alpha)\) is strongly maximal.

  Conversely, suppose that \(x \in \Seqdom\) is strongly maximal.
  Using induction, we will construct for every natural number \(n\), a finite
  sequence \(\sigma_n\) of length \(n\) such that
  \(\sigma_n \preceq \sigma_{n+1} \in x\). Given such finite sequences, we
  define the infinite sequence \(\alpha\) by
  \(\initseg{\alpha}{n} \coloneqq \sigma_n\) and we claim that
  \(x = \iota(\alpha)\). The inclusion \(\iota(\alpha) \subseteq x\) is
  clear. For the other inclusion, suppose that \(\tau \in x\). Then by
  directedness of \(x\), we must have \(\sigma_k = \tau\) where
  \(k = \length(\tau)\). Hence, \(\tau \in \iota(\alpha)\), as desired. We now
  describe how to inductively construct each \(\sigma_n\), starting with the
  empty sequence~\(\sigma_0\).
  Assume that we are given
  \(\sigma_0 \preceq \sigma_1 \preceq \cdots \preceq \sigma_m \in x\) of length
  \(0,1,\dots,m\) for some natural number \(m\).  We construct
  \(\sigma_{m+1} \in x\) of length \(m+1\) extending \(\sigma_m\). Since \(A\)
  is assumed to be inhabited, we can extend \(\sigma_m\) by some element
  \(a \in A\) to get a finite sequence \(\tau\) of length \(m+1\). By
  Lemma~\ref{strongly-maximal-algebraic}, either \(\tau \in x\) or
  \(\dset \tau\) and \(x\) are Hausdorff separated. If \(\tau \in x\), then we
  set \(\sigma_{m+1} \coloneqq \tau\). If \(\dset \tau\) and \(x\) are Hausdorff
  separated, then by Lemma~\ref{Hausdorff-separated-criterion}, there exists a
  finite sequence \(\rho \in x\) such that \(\rho\) and \(\tau\) agree on the
  first \(m\) indices (since \(\sigma_m \in x\)), but disagree on index \(m+1\).
  We now define \(\sigma_{m+1}\) to be the extension of \(\sigma_m\) by
  \(\rho({m+1})\). Notice that \(\sigma_{m+1} \in x\), as desired.
\end{proof}

\begin{proposition}\label{seq-strongly-max-wem}
  Suppose that \(A\) has at least two elements. If every
  maximal element of \(\Seqdom\) is strongly maximal, then weak excluded middle
  holds.
\end{proposition}
\begin{proof}
  Let \(a_0\) and \(a_1\) be two nonequal elements of \(A\).
  We use Lemma~\ref{maximal-strongly-maximal-wem} with the following elements
  \({x \coloneqq \iota(\seq{a_0,a_0,\dots})}\) and
  \(y \coloneqq \iota(\seq{a_1,a_1,\dots})\) of \(\Seqdom\). By
  Theorem~\ref{seq-iota-strongly-maximal}, the elements \(x\) and \(y\) are
  both strongly maximal. Moreover, \(x\) and \(y\) have a greatest lower bound,
  namely the singleton containing the empty sequence~\(\seq{}\). Finally, \(x\) and \(y\)
  are apart, as witnessed by the disjoint Scott opens
  \(\upupset \set{\seq{},\seq{a_0}}\) and \(\upupset \set{\seq{},\seq{a_1}}\).
\end{proof}

\begin{example}\label{not-not-total-maximal}
  Theorem~\ref{seq-iota-strongly-maximal} tells us that functions \(\Nat \to A\)
  give strongly maximal elements of \(\Seqdom\).
  An example of a maximal element that constructively fails to be strongly
  maximal is given by a \(\lnot\lnot\)\nobreakdash-total function. More
  precisely, suppose that \(R \subseteq \Nat \times A\) is a single-valued
  binary relation such that %
  (1) membership of \(R\) is decidable, and
  (2) the relation \(R\) is \(\lnot\lnot\)-total, i.e.\ we have
  \(\lnot\lnot\pa*{\forall_{n \in \Nat}\exists_{a \in A}\,(n,a) \in R}\).
  We define
  \[
    \overline{R} \coloneqq \set{\sigma \in A^\ast \mid (0,\sigma_0) , \ldots ,
      (n-1,\sigma_{n-1}) \in R \text{ with \(n\) the length of \(\sigma\)}}.
  \]
  Then \(\overline R\) is a lower set and semidirected by single-valuedness of
  \(R\), so that \(\overline R \in \Seqdom\).
  Note that, by Theorem~\ref{seq-iota-strongly-maximal}, \(\overline R\) is not
  necessarily strongly maximal as \(R\) is only \(\lnot\lnot\)-total.
  But we claim that \(\overline R\) \emph{is} maximal. For suppose that
  \(x \in \Seqdom\) is such that \(\overline R \subseteq x\); we show that
  \(x \subseteq \overline R\). So let \(\sigma \in x\).
  By our assumption that membership of \(R\) is decidable it suffices to check
  that for every natural number \(i\) less than the length of \(\sigma\) we have
  \(\lnot((i,\sigma_i) \not\in R)\).
  So suppose for a contradiction that there exists an \(i\) with
  \((i,\sigma_i) \not\in R\).
  Since membership of \(R\) is decidable, we may assume that \(i\) is the least
  such number.
  We claim that there is no element \(a \in A\) for which \((i,a) \in R\) holds,
  which contradicts \(\lnot\lnot\)-totality of \(R\).
  For if \(a \in A\) were such an element, then we define the sequence \(\tau\)
  by extending \(\overline{\sigma_{i-1}}\) with \(a\).
  By the choice of \(i\), we have \(\tau \in \overline R\) and hence
  \(\tau \in x\) since \(\overline R \subseteq x\).
  But then \(a = \sigma_i\) by semidirectedness of \(x\), but we had
  \((i,\sigma_i) \not\in R\).
\end{example}

\begin{definition}[Markov's Principle]
  \emph{Markov's Principle} is the assertion that for every infinite binary
  sequence \(\phi\) we have that
  \(\lnot \pa*{\forall_{n \in \Nat}\,\phi(n) = 0}\) implies
  \(\exists_{k \in \Nat}\,\phi(k) =1\).
\end{definition}
Markov's Principle~\citep{BridgesRichman1987} follows from excluded middle, but
is independent in constructive mathematics, i.e.\ Markov's Principle is not
provable and neither is its negation.

\begin{therm}
  For the Cantor domain, the strongly maximal elements are exactly those
  elements that are both sharp and maximal.
  For the Baire domain, the strongly maximal elements are exactly the elements
  that are both sharp and maximal if and only if Markov's Principle holds.
\end{therm}
\begin{proof}
  By Propositions~\ref{sharp-if-strongly-maximal} and
  \ref{maximal-if-strongly-maximal}, every strongly maximal element of a
  continuous dcpo is both sharp and maximal, so it's the assertion that strong
  maximality follows from the conjunction of sharpness and maximality that
  matters.

  Let \(x\) be a sharp and maximal element of the Cantor domain. We prove that
  \(x\) is strongly maximal using Lemma~\ref{strongly-maximal-algebraic}. So let
  \(\sigma\) be a finite binary sequence. We must show that
  \(\dset \sigma \below x\) or \(\dset \sigma\) and \(x\) are Hausdorff
  separated. By sharpness of \(x\) and Proposition~\ref{sharp-algebraic},
  membership of \(x\) is decidable.
  If \(\sigma \in x\), then \(\dset \sigma \below x\) holds and we are done.
  So suppose that \(\sigma \not\in x\). Then using decidability of membership of
  \(x\) and finiteness of \(\sigma\), we can find the shortest prefix \(\tau\)
  of \(\sigma\) for which \(\tau \not\in x\). Assume without loss of generality
  that \(\tau\) ends with \(0\) and write \(\rho\) for the sequence obtained by
  replacing the final \(0\) in \(\tau\) by \(1\). By definition of \(\tau\) and
  the fact that \(x\) is a maximal ideal, it must be the case that
  \(\rho \in x\). But now \(\upupset (\dset \tau)\) and \(\upupset (\dset \rho)\)
  are disjoint Scott opens witnessing Hausdorff separation of \(\dset \sigma\)
  and \(x\). So \(x\) is strongly maximal, as desired.

  Now suppose that every sharp and maximal element of the Baire domain is
  strongly maximal. We show that Markov's Principle follows.
  Let \(\phi\) be an infinite binary sequence for which
  \(\lnot(\forall_{n\in\Nat}\,\phi(n)=0)\) holds.
  We need to find \(k \in \Nat\) such that \(\phi(k) = 1\).
  Define the following element of the Baire domain
  \[
    x_\phi \coloneqq
    \set{\sigma \in \Nat^\ast \mid
      \sigma  \text{ is empty or }
      \sigma \prec \seq{k,k,\dots}
      \text{ and \(k\) is the least \(m \in \Nat\) with \(\phi(m) = 1\)}}.
  \]
  By Proposition~\ref{sharp-algebraic} and the fact that \(x_\phi\) has
  decidable membership, we see that \(x_\phi\) is sharp.
  We proceed to show that \(x_\phi\) is maximal. So suppose that we have \(y\)
  in the Baire domain with \(x_\phi \subseteq y\). We must show that
  \(y \subseteq x_\phi\), so let \(\sigma \in y\) be a finite sequence of
  natural numbers. We are going to show that \(\sigma \not\in x_\phi\) is
  impossible, which suffices, because membership of \(x_\phi\) is decidable.
  So assume for a contradiction that \(\sigma \not\in x_\phi\). Then
  \(\sigma\) must be nonempty and \(\sigma(0)\) is not the least \(m \in \Nat\)
  for which \(\phi(m) = 1\). We are going to show that \(\phi\) must be zero
  everywhere, contradicting our assumption. Let \(n \in \Nat\) be
  arbitrary. Then either \(\phi(n) = 0\) in which case we are done, or
  \(\phi(n) = 1\). We show that the the second case is impossible. For if
  \(\phi(n) = 1\), then we can find \(k \in \Nat\) such that \(k\) is the least
  natural number \(m \in \Nat\) for which \(\phi(m) = 1\). But then
  \(\seq{k} \in x_\phi\), so \(\seq{k} \in y\), but also
  \(\sigma\in y\) while \(\sigma(0) \neq k\), contradicting directedness of
  \(y\).
  Thus, we have proved that \(x_\phi\) is both sharp and maximal. Hence, by
  assumption, the element \(x_\phi\) is strongly maximal.
  By Theorem~\ref{seq-iota-strongly-maximal}, this means that
  \(x_\phi = \iota(\alpha)\) for some infinite sequence \(\alpha\) of natural
  numbers. But then \(\phi(\alpha(0)) = 1\), so \(\alpha(0)\) is the sought
  value.

  Conversely, suppose that Markov's Principle holds and let \(x\) be a sharp and
  maximal element of the Baire domain. We prove that \(x\) is strongly maximal
  using Lemma~\ref{strongly-maximal-algebraic}. So let \(\sigma\) be a finite
  sequence of natural numbers. We must show that \(\dset \sigma \below x\) or
  \(\dset \sigma\) and \(x\) are Hausdorff separated. By sharpness of \(x\) and
  Proposition~\ref{sharp-algebraic}, membership of \(x\) is decidable.
  If \(\sigma \in x\), then \(\dset \sigma \below x\) holds and we are done.
  So suppose that \(\sigma \not\in x\). Then using decidability of membership of
  \(x\) and finiteness of \(\sigma\), we can find the longest prefix \(\tau\)
  of \(\sigma\) for which \(\tau\) is still in \(x\).
  Now, using decidability of membership of \(x\) again, consider the infinite
  binary sequence \(\chi\) given by
  \[
    \chi(n) =
    \begin{cases}
      1 &\text{if \(\tau\) extended by \(n\) is in \(x\)}; \\
      0 &\text{if \(\tau\) extended by \(n\) is not in \(x\)}.
    \end{cases}
  \]
  Notice that \(\chi\) cannot be zero everywhere, as this would contradict the
  assumed maximality of \(x\). Hence, by Markov's Principle, there exists
  \(k \in \Nat\) such that \(\tau\) extended by \(k\) is in \(x\). We
  define \(\rho\) to be this extension of~\(\tau\), so that \(\rho \in x\).
  Then \(\upupset (\dset \sigma)\) and \(\upupset (\dset \rho)\)
  are disjoint Scott opens witnessing Hausdorff separation of \(\dset \sigma\)
  and \(x\). Hence, \(x\) is strongly maximal, as we wished to prove.
\end{proof}

\begin{lemma}\label{seq-T0}
  The space \(\Seqspace\) is \(T_0\)-separated with respect to
  \(\apart_{\Seqspace}\), i.e.\ for \(\alpha,\beta \in \Seqspace\) we have
  \(\alpha \apart_{\Seqspace} \beta\) if and only if there exists an open
  containing \(x\) but not \(y\) or vice versa.
\end{lemma}
\begin{proof}
  If \(\alpha \apart_{\Seqspace} \beta\), then there exists a natural number
  \(n\) such that \(\initseg{\alpha}{n} \neq \initseg{\beta}{n}\) and
  \(\set{\gamma \in \Seqspace \mid \initseg{\alpha}{n} \prec \gamma}\) is an
  open containing \(\alpha\) but not \(\beta\). Conversely, suppose that \(U\)
  is an open containing \(\alpha\) but not \(\beta\). By the description of
  basic opens of \(\Seqspace\), there exists a finite sequence
  \(\sigma\) such that \(\sigma \prec \alpha\), while \(\sigma \nprec \beta\).
  Hence, if we take \(n\) to be the length of \(\sigma\), then
  \(\initseg{\alpha}{n} \neq \initseg{\beta}{n}\), finishing the proof.
\end{proof}

\begin{therm}\label{seq-homeo-apart}
  The map \(\iota\) is a homeomorphism from the space \(\Seqspace\) of infinite
  sequences to the space of strongly maximal elements of the algebraic dcpo
  \(\Seqdom\) with the relative Scott topology.
  Moreover, \(\iota\) preserves and reflects apartness:
  \(\alpha \apart_{\Seqspace} \beta\) if and only if
  \(\iota(\alpha) \apart_{\Seqdom} \iota(\beta)\) for every two infinite
  sequences \(\alpha,\beta \in \Seqspace\).
\end{therm}
\begin{proof}
  Let \(\SMax(\Seqdom)\) be the space of strongly maximal elements of
  \(\Seqdom\) with the relative Scott topology. By
  Theorem~\ref{seq-iota-strongly-maximal}, the map
  \(\iota \colon \Seqspace \to \SMax(\Seqdom)\) is a bijection. Now let
  \(\sigma\) be a finite sequence and consider the basic open
  \(\upupset(\dset \sigma) = \set{x \in \Seqdom \mid \dset \sigma \subseteq x}\)
  of \(\Seqdom\). By Theorem~\ref{seq-iota-strongly-maximal}, we have
  \(\upupset(\dset \sigma) \cap \SMax(\Seqdom) = \set{\iota(\alpha) \mid \sigma
    \prec \alpha}\), which is \(\iota\) applied to a basic open of
  \(\Seqspace\). It follows that \(\iota\) is both open and continuous.
  Finally, note that
  \begin{align*}
    \alpha \apart_{\Seqspace}\beta &\iff
       \exists\, U \subseteq \Seqspace \text{ open separating \(\alpha\) and \(\beta\)}
    &&\text{(by Lemma~\ref{seq-T0})} \\
       &\iff \exists\, U \subseteq \Seqdom \text{ open separating  \(\iota(\alpha)\) and \(\iota(\beta)\)}
    &&\text{(since \(\iota\) is a homeomorphism)} \\
       &\iff \iota(\alpha) \apart \iota(\beta) &&\text{(by definition of the intrinsic apartness)}
  \end{align*}
  as desired.
\end{proof}

Thus, the intrinsic apartness generalizes the well-known apartness on infinite
sequences and in particular those on Cantor and Baire space.

\subsection{Partial Dedekind reals}
The previous section illustrated the theory of the intrinsic apartness and sharp
and strongly maximal elements in the algebraic case. For the continuous case,
the partial Dedekind reals are a very natural example.
We begin by recalling the definition of a (two-sided) Dedekind real number.
\begin{definition}[Dedekind real]\label{def:Dedekind-real}
  Given a pair \(x = (L_x,U_x)\) of subsets of \(\Rat\), we suggestively write
  \(p < x\) for \(p \in L_x\) and \(x < q\) for \(q \in U_x\). A~\emph{Dedekind
    real} \(x\) is a pair \((L_x,U_x)\) of subsets of \(\Rat\) satisfying the
  following properties:
  \begin{enumerate}
  \item \emph{boundedness}: there exist \(p,q \in \Rat\) such that \(p < x\) and
    \(x < q\);
  \item \emph{roundedness}: for every \(p \in \Rat\), we have
    \(p < x \iff \exists_{r \in \Rat} \pa*{p < r} \land \pa*{r < x}\) and
    similarly, for every \(q \in \Rat\), we have
    \(x < q \iff \exists_{s \in \Rat}\pa{s < q} \land \pa{x < s}\);
  \item \emph{transitivity}: for every \(p,q \in \Rat\), if \(p < x\) and
    \(x < q\), then \(p < q\);
  \item \emph{locatedness}: for every \(p,q\in\Rat\) with \(p < q\) we have
    \(p < x\) or \(x < q\).
  \end{enumerate}
\end{definition}

\begin{definition}[Real line \(\Rea\)]
  The \emph{real line} \(\Rea\) is the topological space of all Dedekind real
  numbers whose basic opens are given by
  \(\set{x \in \Rea \mid p < x \text{ and } x < q}\) for \(p,q \in \Rat\). The
  space \(\Rea\) has a natural notion of apartness, namely:
  \(x \apart_{\Rea} y \iff {\exists_{p \in \Rat}\,\pa*{x < p < y} \vee \pa*{y <
        p < x}}\).
\end{definition}

\begin{definition}[Partial Dedekind reals \(\Realdom\)]
  Consider the set
  \(\Rat \times_{<} \Rat \coloneqq \set{(p,q) \in \Rat\times\Rat \mid p < q}\)
  ordered by defining the strict order \((p,q) \prec (r,s) \iff p < r < s <
  q\). The pair \(\pa*{\Rat \times_{<} \Rat , \prec}\) is an abstract basis, so
  \(\Realdom \coloneqq \Idl(\Rat\times_{<}\Rat,\prec)\) is a continuous dcpo and
  we refer to its elements as \emph{partial Dedekind reals}.
\end{definition}

\begin{lemma}\label{Dedekind-way-below-criterion}
  For every two rationals \(p < q\) and \(I \in \Realdom\), we have
  \(\dset(p,q) \ll I\) if and only if \((p,q) \in I\).
  In particular, \(\dset(p,q) \ll \dset(r,s)\) if and only if \(p < r < s < q\).
\end{lemma}
\begin{proof}
  Suppose that \(\dset(p,q) \ll I\). By a well-known fact~\cite[Item~2 of
  Proposition~2.2.22]{AbramskyJung1995} of rounded ideals, there exists
  \((p',q') \in I\) such that \(\dset(p,q) \subseteq \dset(p',q')\). Using
  roundedness, there exists \((r,s) \in I\) such that \(p' < r < s < q'\). We
  claim that \(p < r\) and \(s < q\), from which \((p,q) \in I\) follows as
  \(I\) is a lower set. We use trichotomy on the rationals, so assume for a
  contradiction that \(r \leq p\). Then \(p' < p\), contradicting
  \(\dset(p,q) \subseteq \dset(p',q')\). Thus \(p < r\) and similarly,
  \(s < q\), as desired.  The converse follows directly from \cite[Item~2 of
  Proposition~2.2.22]{AbramskyJung1995}.
\end{proof}

\begin{definition}[\(\iota\)]
  We define an injection \(\iota \colon \Rea \hookrightarrow \Realdom\) by
  \(\iota\pa*{L_x,U_x} \coloneqq \set{(p,q) \mid p \in L_x, q \in U_x}\).
  The map \(\iota\) is well-defined precisely because a Dedekind real is
  required to be bounded, rounded and transitive.
\end{definition}

\begin{therm}\label{Dedekind-iota-strongly-maximal}
  The image of \(\iota\) is exactly the subset of strongly maximal elements of
  \(\Realdom\).
\end{therm}
\begin{proof}
  Suppose that \(x\) is a Dedekind real. We show that \(\iota(x)\) is strongly
  maximal in \(\Realdom\). We use Lemma~\ref{strongly-maximal-basis}, so assume
  that we have rationals \(p,q,r,s \in \Rat\) such that
  \(\dset(p,q) \ll \dset(r,s)\). Then we have \(p < r < s < q\) by
  Lemma~\ref{Dedekind-way-below-criterion}. By locatedness of \(x\), we have
  \(p < x\) or \(x < r\).
  If \(x < r\), then there exist rationals \(u < x < v\) such that
  \(u < v < r\), so that the intervals \((u,v)\) and \((r,s)\) don't
  overlap. Hence, \(\dset (r,s)\) and \(x\) are seen to be Hausdorff separated
  using Lemma~\ref{Hausdorff-separated-criterion}.
  Now suppose that \(p < x\). We use locatedness of \(x\) once more to decide
  whether \(s < x\) or \(x < q\). If \(s < x\), then, similarly to the above, we
  show that \(\dset (r,s)\) and \(x\) are Hausdorff separated. And if \(x < q\),
  then \((p,q) \in \iota(x)\), so \(\dset (p,q) \ll \iota(x)\) by
  Lemma~\ref{Dedekind-way-below-criterion}. Thus, \(\iota(x)\) is strongly
  maximal in \(\Realdom\).

  Conversely, suppose that \(I \in \Realdom\) is strongly maximal.
  Define \(L \coloneqq \set{p \in \Rat \mid \exists_{q \in \Rat}\,(p,q) \in I}\)
  and \(U \coloneqq \set{q \in \Rat \mid \exists_{p \in \Rat}\,(p,q) \in I}\)
  and set \(x = \pa*{L,U}\). It is straightforward to show that \(x\) is
  bounded, rounded and transitive. We show that \(x\) is also located.
  So suppose that we have rationals \(p < q\). Then there exist rationals \(r\)
  and \(s\) such that \(\dset(p,q) \ll \dset(r,s)\), i.e.\ \(p < r < s < q\). By
  strong maximality of \(I\), we have \(\dset (p,q) \ll I\) or \(\dset (r,s)\)
  and \(I\) can be Hausdorff separated. If \(\dset(p,q) \ll I\), then
  \(p \in L\) (and \(q \in U\)) by Lemma~\ref{Dedekind-way-below-criterion}, so
  we are done. Now if \(\dset(r,s)\) and \(I\) are Hausdorff separated, then by
  Lemma~\ref{Hausdorff-separated-criterion}, there exist \(u \in L\) and
  \(v \in U\) such that the intervals \((r,s)\) and \((u,v)\) don't overlap. So
  either \(s < u\) or \(v < r\). If \(s < u\), then \(p < x\), and if \(v < r\),
  then \(x < q\).
  Thus, \(x = \pa*{L,U}\) is located and indeed a Dedekind real.
  Finally, one can verify that \(\iota(x) = I\), as desired.
\end{proof}

With excluded middle, the image of \(\iota\) is just the set of maximal elements
of \(\Realdom\). The following result highlights the constructive strength of
locatedness of Dedekind reals.

\begin{proposition}\label{Dedekind-max-strong-max-wem}
  If every maximal element of \(\Realdom\) is strongly maximal, then weak
  excluded middle holds.
\end{proposition}
\begin{proof}
  We use Lemma~\ref{maximal-strongly-maximal-wem} with the following elements
  \(x \coloneqq \iota(0)\) and \(y \coloneqq \iota(1)\) of \(\Realdom\). By
  Theorem~\ref{Dedekind-iota-strongly-maximal}, the elements \(x\) and \(y\) are
  both strongly maximal. Moreover, \(x\) and \(y\) have a greatest lower bound,
  namely \(\dset(0,1)\). Finally, \(x\) and \(y\) are apart, as witnessed by the
  disjoint Scott opens \(\upupset\pa*{\dset(-1/2,1/2)}\) and
  \(\upupset \pa*{\dset(1/2,3/2)}\).
\end{proof}

We conjecture that \(\Realdom\) is similar to the Baire domain in that the
strongly maximal elements of \(\Realdom\) only coincide with the elements that
are both sharp and maximal if a constructive taboo holds.

\begin{lemma}\label{Dedekind-T0}
  The Dedekind real numbers are \(T_0\)-separated with respect to \(\apart_{\Rea}\),
  i.e.\ for \(x,y \in \Rea\) we have
  \(x \apart_{\Rea} y\) if and only if there exists an open \(U\)
  containing \(x\) but not \(y\) or vice versa.
\end{lemma}
\begin{proof}
  If \(x \apart_\Rea y\), then we may assume without loss of generality that
  there exists \(p \in \Rat\) with \(x < p < y\). Then there also exists
  \(q \in \Rat\) such that \(p < y < q\) and \(\set{z \in \Rea \mid p < z < q}\)
  is an open separating \(x\) and \(y\).
  Conversely, suppose that \(U\) is an open containing \(x\) but not \(y\). By
  description of the basic opens of \(\Rea\), there exists a rationals \(p\) and
  \(q\) with \(p < x < q\), while \(\lnot\pa*{\pa*{p < y} \land \pa*{y < q}
  }\). Now find \(r,s \in \Rat\) such that \(p < r < x < s < q\). Using
  locatedness of \(y\), we have \(p < y\) or \(y < r\). In the second case we
  see that \(x \apart_\Rea y\), as desired. And if \(p < y\), then as \(y\) is
  located and \(y < q\) is now impossible, we must have \(s < y\), so that
  \(x \apart_\Rea y\) again.
\end{proof}

\begin{therm}\label{Dedekind-iota-homeo}
  The map \(\iota\) is a homeomorphism from \(\Rea\) to the space of strongly
  maximal elements of the continuous dcpo \(\Realdom\) with the relative Scott
  topology.
  Moreover, \(\iota\) preserves and reflects apartness.
\end{therm}
\begin{proof}
  Let \(\SMax(\Realdom)\) be the space of strongly maximal elements of
  \(\Realdom\) with the relative Scott topology. By
  Theorem~\ref{Dedekind-iota-strongly-maximal}, the map
  \(\iota \colon \Rea \to \SMax(\Realdom)\) is a bijection. Now let \(p < q\) be
  rationals and consider the basic open
  \(
  \upupset(\dset (p,q))\) of \(\Realdom\).
  By Lemma~\ref{Dedekind-way-below-criterion}, this basic open is equal to
  \(\set{I \in \Realdom \mid (p,q) \in I}\).
  So Theorem~\ref{Dedekind-iota-strongly-maximal} tells us
  \(\upupset(\dset (p,q)) \cap \SMax(\Realdom) = \set{\iota(x) \mid p < x <
    q}\), which is \(\iota\) applied to a basic open of \(\Rea\). It follows
  that \(\iota\) is both open and continuous.
  Finally, note that
  \begin{align*}
    x \apart_{\Rea}y &\iff \exists\, U \subseteq \Rea \text{ open separating \(x\) and \(y\)}
    &&\text{(by Lemma~\ref{Dedekind-T0})} \\
    &\iff \exists\, U \subseteq \Realdom \text{ open separating \(\iota(x)\) and \(\iota(y)\)}
    &&\text{(since \(\iota\) is a homeomorphism)} \\
    &\iff \iota(x) \apart \iota(y) &&\text{(by definition of the intrinsic apartness)}
  \end{align*}
  as desired.
\end{proof}

\subsection{Lower reals}
We now consider lower reals, which feature a nice illustration of sharpness.
\begin{definition}[Lower reals \(\LowerReal\)]
  The pair \(\pa*{\Rat , <}\) is an abstract basis, so
  \(\LowerReal \coloneqq \Idl(\Rat,<)\) is a continuous dcpo and we refer to its
  elements as \emph{lower reals}.
\end{definition}

\begin{lemma}\label{lower-reals-way-below-criterion}
  For every \(p \in \Rat\) and \(L \in \LowerReal\), we have
  \(\dset p \ll L\) if and only if \(p \in L\).
\end{lemma}
\begin{proof}
  Similar to Lemma~\ref{Dedekind-way-below-criterion}.
\end{proof}

\begin{lemma}\label{U-from-L}
  If \(L \in \LowerReal\) is a lower real, then the pair \((L,U)\) with
  \(U \coloneqq \set*{q \in \Rat \mid \exists_{s \in \Rat \setminus L}\,s < q}\)
  is rounded and transitive in the sense
  of~Definition~\ref{def:Dedekind-real}.
  Moreover, if \(\Rat \setminus L\) is inhabited, then \((L,U)\) is bounded too.
\end{lemma}
\begin{proof}
  Let \(L \in \LowerReal\) be a lower real and let \(U\) be as in the lemma.  We
  claim that \(U \subseteq \Rat \setminus L\). For if \(q \in U\), then there
  exists \(s \in \Rat \setminus L\) with \(s < q \). But \(L\) is a lower set,
  so \(q \in L\) would imply \(s \in L\), contradicting that \(s \not\in L\).
  We first prove transitivity. Suppose that \(p \in L\) and \(q \in U\). By
  trichotomy on the rationals, it suffices to prove that \(q \leq p\) is
  impossible. So assume for a contradiction that \(q \leq p\). Then \(q \in L\),
  because \(L\) is a lower set. But \(q \in U \subseteq \Rat \setminus L\), so
  \(q \not\in L\), contradicting \(q \in L\).
  For roundedness, observe that
  \(p \in L \iff \exists_{r \in \Rat}\pa*{p < r} \land \pa*{r \in L}\), because
  \(L\) is a rounded ideal.
  Now suppose that \(q \in U\). Then there exists \(s \in \Rat \setminus L\)
  with \(s < q\). Now find \(r \in \Rat\) such that \(s < r < q\) and we see
  that \(r \in U\).
  Conversely, if we have \(s \in \Rat\) with \(s < q\) and \(s \in U\), then
  \(q \in U\), because \(s \in U \subseteq \Rat\setminus L\).
  Hence, \(q \in U \iff \exists_{s \in \Rat}\pa{s < q} \land \pa{s \in U}\), so
  \((L,U)\) is rounded.
  Finally, if \(\Rat \setminus L\) is inhabited, then \((L,U)\) is bounded,
  because \(L\) is inhabited too, as it is directed.
\end{proof}

Classically, every lower real whose complement is inhabited determines a
Dedekind real by the construction above. It is well-known that constructively a
lower real may fail to be located. The following result offers a
domain-theoretic explanation of that phenomenon.

\begin{therm}\label{sharp-iff-located}
  A lower real \(L \in \LowerReal\) is sharp if and only if the pair \((L,U)\)
  with \(U\)~as in Lemma~\ref{U-from-L} is located.
\end{therm}
\begin{proof}
  Suppose that \(L \in \LowerReal\) is sharp and let \(p < q\) be rationals.
  Find \(r \in \Rat\) such that \(p < r < q\). By
  Lemma~\ref{lower-reals-way-below-criterion}, we have
  \(\dset p \ll \dset r \ll \dset q\). By sharpness, we have
  \(\dset p \ll L\) or \(\dset r \not\subseteq L\). In the first case,
  \(p \in L\) and we are done; and if \(\dset r \not\subseteq L\), then
  \(r \not\in L\), so \(q \in U\). Hence, \((L,U)\) is located.
  For the converse, assume that \((L,U)\) is located. We use
  Lemma~\ref{sharp-basis} to prove that \(L\) is sharp. So let \(p,q \in \Rat\)
  with \(\dset p \ll \dset q\). By Lemma~\ref{lower-reals-way-below-criterion},
  this yields \(p < q\). By locatedness, \(p \in L\) or \(q \in U\). If
  \(p \in L\), then \(\dset p \ll L\) and we are done; and if \(q \in U\), then
  we have \(s \in \Rat \setminus L\) with \(s < q\) so that
  \(\dset q \not\subseteq L\). Hence, \(L\) is sharp, as desired.
\end{proof}

\subsection{An alternative domain for Cantor space}\label{lifting-Cantor}
An alternative domain for Cantor space is given by embedding Cantor space
\(\Two^\Nat\) into the exponential of free pointed dcpos
\(\lifting(\Two)^{\lifting(\Nat)}\); we study sharpness for this domain.

Classically, the free pointed dcpo on a set \(X\) is given by the flat dcpo
\(X \cup \set{\bot}\). Constructively, we use the
\emph{lifting}~\citep{EscardoKnapp2017} of \(X\) and we denote it by
\(\lifting(X)\). We can explicitly describe the elements of \(\lifting(X)\) as
partial maps from a singleton to \(X\), but for our present purposes it will be
easier to work with the lifting abstractly and only use the following
properties, which were proved in~\citep{deJongEscardo2021a}
and~\citep{deJong2022}:

\begin{itemize}
\item The unit \(\eta \colon X \to \lifting(X)\) is injective for every set
  \(X\).
\item The elements in the image of \(\eta\) are all incomparable in the order of
  \(\lifting(X)\).
\item For every map of sets \(f \colon X \to Y\), the functor \(\lifting\)
  yields a Scott continuous function \(\lifting(f)\) that is \emph{strict},
  i.e.\ \(\lifting(f)\) preserves the least element.
  Also, if \(\lifting(f) = \lifting(g)\), then \(f = g\) by injectivity and
  naturality of \(\eta\).
\item The lifting \(\lifting(X)\) is algebraic and bounded complete for every
  set \(X\). Its compact elements are given by
  \(\set{\bot} \cup \set{\eta(x) \mid x \in X}\).
\end{itemize}

\begin{definition}[\(\varepsilon\)]
  The lifting functor \(\lifting\) defines an injection
  \(\varepsilon \colon \Two^\Nat \hookrightarrow
  \lifting(\Two)^{\lifting(\Nat)}\) from Cantor space into the exponential dcpo.
\end{definition}

The exponential \(\lifting(\Two)^{\lifting(\Nat)}\) is important in higher-type
computability~\citep{Escardo2008} for instance.
What is noteworthy about \(\lifting(\Two)^{\lifting(\Nat)}\) is that, unlike in
Section~\ref{sec:Cantor-and-Baire-domains}, it are not the (strongly) maximal
elements of \(\lifting(\Two)^{\lifting(\Nat)}\) that matter, but the sharp
elements still play an important role.

\begin{therm}
  Every element in the image of \(\varepsilon\) is sharp, but not all of them
  are maximal. Also, not every sharp element is in the image of \(\varepsilon\).
\end{therm}
\begin{proof}
  Let \(\alpha \in \Two^\Nat\) be arbitrary. We wish to show that
  \(\varepsilon(\alpha)\) is sharp. By
  Propositions~\ref{dec-closed-under-exponentials} and~\ref{sharp-algebraic}, it
  suffices to show that \(\steppa{a \To b} \below \varepsilon(\alpha)\) is
  decidable for elements \(a \in \set{\bot}\cup\set{\eta(n) \mid n \in \Nat}\) and
  elements \(b \in \set{\bot}\cup\set{\eta(i) \mid i \in \Two}\).  By
  Lemma~\ref{step-function-below}, this reduces to proving decidability of
  \(b \below \varepsilon(\alpha)(a)\) for such elements \(a\) and \(b\).
  If \(b = \bot\), then the inequality certainly holds.
  So suppose that \(b = \eta(i)\) for some \(i \in \Two\).
  If \(a = \bot\), then \(\varepsilon(\alpha)(\bot) = \bot\) by strictness of
  \(\varepsilon(\alpha)\), so \(b = \eta(i) \below \varepsilon(\alpha)(a)\) is
  false.
  And finally, if \(a = \eta(n)\) for some \(n \in \Nat\), then
  \(\varepsilon(\alpha)(a) = \varepsilon(\alpha)(\eta(n)) \equiv
  \lifting(\alpha)(\eta(n)) = \eta(\alpha(n))\), by naturality of \(\eta\). So
  we need to decide \(\eta(i) \below \eta(\alpha(n))\), which is equivalent to
  \(\eta(i) = \eta(\alpha(n))\), because elements in the image of \(\eta\) are
  all incomparable. And \(\eta(i) = \eta(\alpha(n))\) holds if and only if
  \(i = \alpha(n)\) does, because \(\eta\) is injective. But \(i = \alpha(n)\)
  is decidable, because \(\Two\) has decidable equality.

  For the second claim, fix an element \(i \in \Two\) and consider the constant
  map \(n \mapsto i\) as an element of \(\Two^\Nat\). Notice that the
  inequality \(\varepsilon(n \mapsto i) \below \pa*{x \mapsto \eta(i)}\) holds
  in \(\lifting(\Two)^{\lifting(\Nat)}\). But \(\varepsilon(n \mapsto i)\) is
  strict, while \(x \mapsto \eta(i)\) is not, so \(\varepsilon(n \mapsto i)\)
  and \(x \mapsto \eta(i)\) are not equal and hence,
  \(\varepsilon(n \mapsto i)\) is not maximal.

  Finally, one can check that, for any \(i \in \Two\), the element
  \(x \mapsto \eta(i)\) in \(\lifting(\Two)^{\lifting(\Nat)}\) is sharp (in
  fact, it is strongly maximal). But it's not strict, so it cannot be in the
  image of \(\varepsilon\).
\end{proof}

\begin{therm}\label{alt-Cantor-homeo}
  The map \(\varepsilon\) is a homeomorphism from Cantor space to the image of
  \(\varepsilon\). Moreover, \(\varepsilon\) preserves and reflects apartness.
\end{therm}
\begin{proof}
  The basic opens of \(\lifting(\Two)^{\lifting(\Nat)}\) are of the form
  \(\upupset s = \upset s = \set{f \in \lifting(\Two)^{\lifting(\Nat)} \mid s
    \below f}\) with \(s\) the join of some Kuratowski finite subset
  \(\set{a_i \To b_i \mid 0 \leq i \leq n-1}\) of single-step functions, where
  \(a_k \in \set{\bot}\cup\set{\eta(n) \mid n \in \Nat}\) and
  \(b_k \in \set{\bot}\cup\set{\eta(i) \mid i \in \Two}\).
  By Lemma~\ref{step-function-below}, we have \(s \below f\) for a function
  \(f \in \lifting(\Two)^{\lifting(\Nat)}\) if and only if \(b_k \below f(a_k)\)
  for every \(0 \leq k \leq {n-1}\) %
  if and only if \(i_k = \alpha(n_k)\) for every such \(k\), where
  \(\beta_k = \eta(i_k)\).
  Now define the set of indices
  \(K \coloneqq \set {0 \leq k \leq n-1 \mid a_k,b_k \neq \bot}\).
  Because \(\varepsilon(\alpha)\) is strict for every \(\alpha \in \Two^\Nat\),
  we see that \(s \below \varepsilon(\alpha)\) if and only if
  \(b_k \below \varepsilon(\alpha)(a_k)\) for every \(k \in K\).
  If \(i \in \Two\) and \(n \in \Nat\), then
  \(\eta(i) \below \varepsilon(\alpha)(\eta(n))\) holds if and only if
  \(i = \alpha(n)\), by naturality of \(\varepsilon\), injectivity of \(\eta\),
  and the fact that elements in the image of \(\eta\) are incomparable.
  Hence, the set \(\upupset s \cap \image(\varepsilon)\) is seen to be
  \(\varepsilon\) applied to a basic open of Cantor space.
  It follows that \(\varepsilon\) is open and continuous.
  The proof that \(\varepsilon\) preserves and reflects apartness is similar to
  the proof of the second part of Theorem~\ref{seq-homeo-apart}.
\end{proof}

\nocite{Escardo2008}

\section{Conclusion}\label{sec:conclusion}
Working constructively, we studied continuous dcpos and the Scott topology and
introduced notions of intrinsic apartness and sharp elements.
We showed that our apartness relation is particularly well-suited for continuous
dcpos that have a basis satisfying certain decidability conditions, which hold in
examples of interest. For instance, for such continuous dcpos, the
Bridges--{V\^i\c{t}\v{a}} apartness topology and the Scott topology coincide.
We proved that no apartness on a nontrivial dcpo can be cotransitive
or tight unless (weak) excluded middle holds. But the intrinsic apartness is
tight and cotransitive when restricted to sharp elements.
If a continuous dcpos has a basis satisfying the previously mentioned
decidability conditions, then every basis element is sharp. Another class of
examples of sharp elements is given by the strongly maximal elements.
In fact, strong maximality is closely connected to sharpness and the Lawson
topology. For example, an element \(x\) is strongly maximal if and only if \(x\)
is sharp and every Lawson neighbourhood of~\(x\) contains a Scott neighbourhood
of~\(x\).
Finally, we presented several natural examples of continuous dcpos that
illustrated the intrinsic apartness, strong maximality and sharpness.

In future work, it would be interesting to explore whether a constructive and
predicative treatment is possible, in particular, in univalent foundations
without Voevodsky's resizing axioms as in~\citep{deJongEscardo2021a}. Steve
Vickers also pointed out two directions for future research. The first is to
consider formal ball domains~\cite[Example~V-6.8]{GierzEtAl2003}, which may
subsume the partial Dedekind reals example. The second is to explore the
ramifications of Vickers' observation that refinability
(Definition~\ref{def:refine}) is decidable, even when the order is not, if the
dcpo is algebraic and 2/3~SFP~\cite[p.~157]{Vickers1989}.
Related to the examples, there is still the question of whether we can derive a
constructive taboo from the assumption that strong maximality of a partial
Dedekind real follows from having both sharpness and maximality, as discussed
right after Proposition~\ref{Dedekind-max-strong-max-wem}.
Finally, the Lawson topology deserves further investigation within a
constructive framework.

\section*{Financial Support} This research received no specific grant from any
funding agency, commercial or not-for-profit sectors.

\section*{Competing interests} The author declares none.


\end{document}